\documentclass[11pt]{article}
\usepackage{amsthm,amssymb,amsmath,xcolor}
\usepackage{nccmath}
\usepackage{fullpage}
\RequirePackage{marginnote,hyperref}
\usepackage{cleveref}
\allowdisplaybreaks

\newcommand{\aside}[1]{\marginnote{\scriptsize{#1}}[0cm]}
\newcommand{\aaside}[2]{\marginnote{\scriptsize{#1}}[#2]}
\newcommand\Emph[1]{\emph{#1}\aside{#1}}
\newcommand\EmphE[2]{\emph{#1}\aaside{#1}{#2}}

\def\aftermath{\par\vspace{-\belowdisplayskip}\vspace{-\parskip}\vspace{-\baselineskip}}

\def\vph{\varphi}
\def\chipcf{\chi_{\textrm{pcf}}}
\def\Z{\mathbb{Z}}
\def\PP{\mathcal{P}}

\def\Hy{\mathcal{H}}
\def\rank{{\rm{rank}}}
\def\mr{{\rm{mr}}}
\renewcommand\deg{d}

\newtheorem{lem}{Lemma}
\newtheorem{lemA}{Lemma}

\newtheorem{cor}[lem]{Corollary}

\newtheorem{conj}[lem]{Conjecture}
\newtheorem{theorem}[lem]{Theorem}
\newtheorem{prop}[lem]{Proposition}
\renewcommand{\ge}{\geqslant}
\renewcommand{\le}{\leqslant}
\renewcommand{\geq}{\geqslant}
\renewcommand{\leq}{\leqslant}
\renewcommand{\epsilon}{\varepsilon}

\begin{document}
\author{Daniel W. Cranston\thanks{%
Department of Computer Science, Virginia Commonwealth University, Richmond, VA, USA;
\texttt{dcranston@vcu.edu}.
}
\and Chun-Hung Liu\thanks{Department of Mathematics, Texas A\&M University, College Station, TX, USA; \texttt{chliu@tamu.edu}. 
Partially supported by NSF under award DMS-1954054 and CAREER award DMS-2144042.}}
\title{Proper Conflict-free Coloring of\\ Graphs with Large Maximum Degree}
\maketitle
\begin{abstract}
A proper coloring of a graph is \emph{conflict-free} if, for every non-isolated
vertex, some color is used exactly once on its neighborhood.  Caro,
Petru\v{s}evski, and \v{S}krekovski proved that every graph $G$ has a proper
conflict-free coloring with at most $5\Delta(G)/2$ colors and conjectured that
$\Delta(G)+1$ colors suffice for every connected graph $G$ with $\Delta(G)\ge 3$.
Our first main result is that even for list-coloring, $\left\lceil
1.6550826\Delta(G)+\sqrt{\Delta(G)}\right\rceil$ colors suffice for every
graph $G$ with $\Delta(G)\ge 10^{8}$; we also prove slightly weaker bounds
for all graphs with $\Delta(G)\ge 750$.  These results follow from our more
general framework on proper conflict-free list-coloring of a pair consisting of
a graph $G$ and a ``conflict'' hypergraph $\Hy$.  As another corollary of our
results in this general framework, every graph has a proper $(\sqrt{30}+o(1))\Delta(G)^{1.5}$-list-coloring such that every bi-chromatic component is a path on at most three vertices, where the number of colors is optimal up to a constant factor.  
Our proof uses a fairly new type of recursive counting argument called Rosenfeld counting, which is a variant of the Lov\'{a}sz Local Lemma or entropy compression.

We also prove an asymptotically optimal result for a fractional analogue of our
general framework for proper conflict-free coloring for pairs of a graph and a
conflict hypergraph.  A corollary states that every graph $G$ has a fractional
$(1+o(1))\Delta(G)$-coloring such that every fractionally bi-chromatic
component has at most two vertices.  In particular, it implies that the
fractional analogue of the conjecture of Caro et al. holds asymptotically in a
strong sense.	
\end{abstract}

\section{Introduction}
Motivated by a frequency assignment problem in cellular networks, Even, Lotker,
Ron, and Smorodinsky \cite{ELRS} introduced conflict-free coloring of hypergraphs.
A \EmphE{{coloring}}{0mm} of a graph or a hypergraph $G$ is a map $\vph\colon V(G)\to \Z^+$.
A coloring $\vph$ of a hypergraph $\Hy$ is \EmphE{{conflict-free}}{0mm} if for
every (non-empty)\footnote{In this paper, we assume that every edge of a
hypergraph is non-empty, though the edge set can be empty.} $e \in E(\Hy)$,
there exists a color that is used exactly once by $\vph$ on $e$.  Pach and
Tardos \cite{PT} studied this notion and proved that every hypergraph with fewer
than ${s \choose 2}$ edges (for some integer $s$) has a conflict-free coloring
with fewer than $s$ colors.  Note that being conflict-free on
an edge of size 2 is equivalent to the vertices in this edge using
distinct colors.  Hence, the result of Pach and Tardos is optimal, as witnessed by
complete 2-uniform hypergraphs.  Kostochka, Kumbhat, and \L uczak \cite{KKL}
further studied conflict-free coloring for uniform hypergraphs.

A coloring $\vph$ of a graph $G$ is \EmphE{proper}{0mm} if $\vph(v)\ne \vph(w)$
for all $vw\in E(G)$.  As we have seen, being conflict-free on every edge of a
hypergraph with size 2 is equivalent to being a proper coloring of a graph.  So
it is more convenient to consider the following notion, so that we can focus on
edges with larger size.  For a graph $G$ and a hypergraph $\Hy$ with
$V(G)=V(\Hy)$, a \EmphE{{proper conflict-free coloring of $(G,\Hy)$}}{0mm} is a
proper coloring of $G$ that is
also a conflict-free coloring of $\Hy$.  This notion is general: given a graph
$G$, by appropriately defining
 an associated hypergraph $\Hy$, every proper
conflict-free coloring of $(G,\Hy)$ is an acyclic coloring, a star coloring,
and a frugal coloring of $G$, respectively.\footnote{Always $V(\Hy)=V(G)$ and for acyclic, star, and $k$-frugal coloring we add an edge to $\Hy$ with vertex $S$, respectively, when $S$ spans a cycle, $S$ spans a $P_4$, or $S$ is a subset of size $k+1$ of the open neighborhood of some vertex in $G$.}
(Note that proper conflict-free coloring is not equivalent to these other notions, but is in fact strictly more general.) 
So an upper bound for the number of colors used in a proper conflict-free coloring for $(G,\Hy)$ provides an upper bound for those extensively studied notions in graph coloring.
On the other hand, for acyclic coloring \cite{AMR}, star coloring \cite{FRR},
and $k$-frugal coloring (for fixed $k$) \cite{HMR}, the numbers of required
colors are known to be superlinear in the graph's maximum degree.  So an
upper bound for proper conflict-free coloring for a general pair $(G,\Hy)$
cannot be linear in the maximum degree of $G$.

In this paper we study sufficient conditions to have 
a proper conflict-free coloring for $(G,\Hy)$ with a number of colors that is
a linear in the maximum degree of $G$.
One such result is related to conflict-free proper coloring of graphs, which
was introduced by Fabrici, Lu\v{z}ar, Rindo\v{s}ov\'{a}, and
Sot\'{a}k~\cite{FLRS} and was further studied in~\cite{CPS,CCKP,hickingbotham,liu}.
For a graph $G$, a \EmphE{{proper conflict-free coloring of $G$}}{0mm} is a
proper conflict-free coloring of the pair $(G,\Hy)$, where $\Hy$ is the
hypergraph with $V(\Hy)=V(G)$ and the edges of $\Hy$ are the (open)
neighborhoods of the non-isolated vertices of $G$.  In other words, a proper
conflict-free coloring of a graph $G$ is a proper coloring of $G$ such that for
every non-isolated vertex $v$, some color appears exactly once on the neighbors
of $v$.  This notion is a combination of proper coloring and the pointed
conflict-free chromatic parameter studied in \cite{C,PT}.

For a graph or hypergraph $G$, the \EmphE{degree}{0mm} of a vertex $v$ in $G$,
denoted by \EmphE{$\deg_G(v)$}{4mm}, is the number of edges of $G$ containing
$v$, and we denote by \EmphE{$\Delta(G)$}{4mm} the maximum degree of $G$.
For a graph $G$, we denote by \EmphE{$\chipcf(G)$}{4mm} the minimum $k$ such
that $G$ has a proper conflict-free coloring with $k$ colors.
Caro, Petru\v{s}evski, and \v{S}krekovski~\cite{CPS} proposed the following
conjecture. 

\begin{conj}[\cite{CPS}]  \label{max_deg_conj}
$\chipcf(G)\le \Delta(G)+1$ for every connected graph $G$ with $\Delta(G)\ge 3$.
\end{conj}

The condition $\Delta(G) \geq 3$ in Conjecture \ref{max_deg_conj} is required since $\chipcf(C_5)=5$, but if the conjecture is true, then the condition for connectivity can be removed when $\Delta(G) \geq 4$.
The case $\Delta(G)=3$ of Conjecture~\ref{max_deg_conj} follows from an earlier
result of the second author and Yu \cite[Theorem 2]{LY}, even for the
list-coloring setting.

As a first step toward their conjecture, Caro, Petru\v{s}evski, and \v{S}krekovski~\cite{CPS} proved that $\chipcf(G)\le 5\Delta(G)/2$.
In fact, we can prove $\chipcf(G) \leq 2\Delta(G)+1$ by a simple greedy algorithm (see Proposition \ref{easy_upper} below).  
A goal of this paper is to make further progress toward this conjecture.\footnote{When a version of this paper was under review, the second author and Reed \cite{LR} proved that $\chipcf(G) \leq (1+o(1))\Delta(G)$, so Conjecture \ref{max_deg_conj} holds asymptotically. This bound in \cite{LR} is quantitatively stronger than the bounds in this paper. However, all results in this paper work for list-coloring or for proper conflict-free coloring of pairs of graphs and hypergraphs $(G,\Hy)$. The
result and proof in \cite{LR} do not work for those more general settings.}

Our first result works for list-coloring.  
A \EmphE{list-assignment}{0mm} $L$ for a graph $G$ assigns to each vertex $v\in V(G)$ a list $L(v)$ of allowable colors.  
An \EmphE{$a$-assignment}{0mm}, for some real number $a$, is a list-assignment $L$ such that $|L(v)| \geq a$ for all vertices $v$.  
An \EmphE{$L$-coloring}{0mm} of $G$ is a coloring $\vph$ such that $\vph(v)\in L(v)$ for all $v\in V(G)$.  
We prove the following.

\begin{theorem}\label{main_graph}
Fix a positive integer $\Delta\ge 6.5 \cdot 10^7$, fix a real number $\beta$ with $\Delta \geq \beta\ge 0.6550826\Delta$, and let $a:=\left\lceil\Delta+\beta+\sqrt{\Delta}\right\rceil$.
If $G$ is a graph with maximum degree at most $\Delta$ and $L$ is an $a$-assignment for $G$, then there are at least $\beta^{|V(G)|}$ proper conflict-free $L$-colorings of $G$.
Analogous statements hold when $\Delta\ge 4000$ and $\Delta \geq \beta\ge 0.\overline{6}\Delta$ and when $\Delta\ge 750$ and $\Delta \geq \beta\ge 0.8\Delta$.
\end{theorem}

It is obvious that the term $\sqrt{\Delta}$ in Theorem \ref{main_graph} can be
replaced by $10^{-10}\Delta$ when $\Delta$ is sufficiently large.
We choose $\sqrt{\Delta}$ as the error term because it is a natural sublinear
term and it enables us to only require a reasonably small lower bound on $\Delta$.
We put a small amount of effort into optimizing the lower bound for $\Delta$ for
the cases $\beta \geq 0.\overline{6}\Delta$ and $\beta \geq 0.8\Delta$.
The proofs of the three cases $\beta \geq 0.6550826\Delta$, $\beta \geq
0.\overline{6}\Delta$, and $\beta \geq 0.8\Delta$ are essentially the same, and
the proof can be simplified if we do not care about keeping the lower bound
on $\Delta$ small.
In fact, we prove a more general result (Theorem~\ref{hyper_main}, below) about
proper conflict-free coloring for pairs of graphs and hypergraphs, and
Theorem~\ref{main_graph} follows as a special case for $\Delta$ sufficiently
large.

Recall that we mentioned a greedy upper bound of $2\Delta(G)+1$ for Conjecture \ref{max_deg_conj}.  It is obtained by the following simple observation, which is a modification of an observation by Pach and Tardos~\cite{PT} on conflict-free coloring for hypergraphs, since the hypergraph $\Hy$ associated with proper conflict-free coloring for $G$ has edge set equal to the set of neighborhoods of non-isolated vertices and hence has $\Delta(\Hy)=\Delta(G)$.

\begin{prop} \label{easy_upper}
Let $G$ be a graph and $\Hy$ be a hypergraph with $V(\Hy)=V(G)$.
If $G$ is $d$-degenerate, then $(G,\Hy)$ has a proper conflict-free coloring
with at most $d+\Delta(\Hy)+1$ colors.
\end{prop}

\begin{proof}
Since $G$ is $d$-degenerate, there exists an ordering $v_1,v_2,...$ of $V(G)$ 
such that for every $i$, there are at most $d$ indices $j$ with $j<i$ such that
$v_iv_j \in E(G)$.  For each edge $f$ of $\Hy$, let $v_f$ be the vertex in $f$
with the smallest index.  We color $v_1,v_2,...$ greedily in the order listed.
For each $i$, when we color $v_i$, we avoid the colors used on all colored
neighbors $v_j$ of $v_i$, and for each edge $f$ of $\Hy$ containing $v_i$ with
$v_i \neq v_f$, we also avoid the color of $v_f$.  Since there are at most $d$
$v_j$'s and at most $\Delta(\Hy)$ $v_f$'s, we only have to avoid at most
$d+\Delta(\Hy)$ colors, so $d+\Delta(\Hy)+1$ colors suffice.  Moreover, this greedy
coloring is clearly proper and is conflict-free for $(G,\Hy)$ since the color
assigned to $v_f$ has a unique occurrence in $f$ for each $f \in E(\Hy)$.
\end{proof}

Proposition \ref{easy_upper} shows that $\Delta(\Hy)$ plays a role for upper
bounding the number of colors for proper conflict-free colorings for $(G,\Hy)$.
Our more general Theorem \ref{hyper_main} shows that the importance of
$\Delta(\Hy)$ is somehow secondary to the size of edges of $\Hy$.
We need some terminology to state Theorem \ref{hyper_main}.  Let $\Hy$ be a
hypergraph.  The \emph{rank} of $\Hy$, denoted by
\Emph{$\rank(\Hy)$}, is $\max_{f \in E(\Hy)}|f|$.  For every vertex $v$
of $\Hy$, we define \EmphE{$\mr_\Hy(v)$}{0mm} to be $\min_{f \in E(\Hy), f \ni v}|f|$.

\begin{theorem} \label{hyper_main}
Let $R$ be a positive integer.  
Let $G$ be a graph and $\Hy$ be a hypergraph with $V(G)=V(\Hy)$ and $\rank(\Hy) \leq R$ such that $|f| \geq 3$ for every $f \in E(\Hy)$.  Let $\beta$ be a real number with $0.6550826 R \leq \beta \leq R$.  
Let $$a:=\lceil \Delta(G)+\beta+ \max_{v \in V(G)}\big(\deg_\Hy(v) \cdot b(v)\big) \rceil,$$
where $$b(v):=\max\{2\beta^{1-\lceil \frac{\mr_\Hy(v)}{2} \rceil}(\log R)^{2\lceil \frac{\mr_\Hy(v)}{2} \rceil},(1-10^{-8})^{(\log R)^{2}}\}.$$
If\footnote{In the first version of this paper \cite{CL}, we proved the same result but only required $R \geq e^{5 \cdot 10^6}$ with a more complicated proof. We elect to present a simpler proof suggested by a referee in this version even though the lower bound for $R$ is weaker.} $R \geq e^{3.1 \cdot 10^8}$, then for every $a$-assignment $L$ of $G$, there are at least $\beta^{|V(G)|}$ proper conflict-free $L$-colorings of $(G,\Hy)$.
\end{theorem}

The condition $|f| \geq 3$ for every $f \in E(\Hy)$ in Theorem \ref{hyper_main}
is mild, since we may move edges of $\Hy$ with size 2 to edges of $G$.
However, this operation might increase $\Delta(G)$.  When considering proper
conflict-free coloring for a graph $G$, this condition $|f| \geq 3$ is
satisfied only when $G$ has minimum degree at least 3.  But vertices of degree
at most 2 can be handled by a simple argument, so Theorem \ref{main_graph}
can be deduced (when $R$ is sufficiently large) from 
Theorem~\ref{hyper_main} without moving edges of size 2 from $\Hy$ to $G$.  

The emphasis of Theorem~\ref{hyper_main} is on making the lower bound hypothesis on $\beta$ as weak as possible; that is, 
making the coefficient on $R$ as small as possible.
So this result is more effective when $\rank(\Hy)$ is large.  
When $\rank(\Hy)$ is small, we can obtain the following result (Theorem \ref{hyper_fixed_rank}) by slightly modifying the proof of Theorem~\ref{hyper_main} to drop logarithmic factors.
We will show that the number of colors mentioned in Theorem \ref{hyper_fixed_rank} is optimal up to a constant factor.

\begin{theorem} \label{hyper_fixed_rank}
Let $r,\epsilon,R$ be positive real numbers. 
Let $G$ be a graph and $\Hy$ be a hypergraph with $V(G)=V(\Hy)$ and $\rank(\Hy) \leq r$ such that $|f| \geq 3$ for every $f \in E(\Hy)$.  
If either 
	\begin{itemize}
		\item $R \geq (1+\frac{1}{\epsilon})r$ and $$a:=\lceil \Delta(G)+R+ (1+\epsilon) \cdot \max_{v \in V(G)}\{\deg_\Hy(v) \cdot r^{\lceil \frac{\mr_\Hy(v)}{2} \rceil}R^{1-\lceil \frac{\mr_\Hy(v)}{2} \rceil}\} \rceil,$$ or 
		\item $r \leq 4$ and $$a:=\lceil \Delta(G)+R+ \Delta(\Hy) \cdot (3R^{-1}+R^{-2})\rceil,$$
	\end{itemize}
then for every $a$-assignment $L$ of $G$, there are at least $R^{|V(G)|}$ proper conflict-free $L$-colorings of $(G,\Hy)$.
\end{theorem}

Theorem \ref{hyper_fixed_rank} has applications to other coloring parameters studied in the literature. 
Given a graph $G$, if we define $\Hy$ to be the hypergraph with $V(\Hy)=V(G)$ such that $E(\Hy)$ consists of the vertex sets of any 4-vertex path (not necessarily induced) in $G$ and the 3-element subsets of $N(v)$ for each vertex $v$ of $G$, then it is easy to show that $\deg_\Hy(v) \leq 2.5\Delta(G)^3$, $\rank(\Hy) \leq 4$ and $\mr_\Hy(v) \geq 3$ for every $v \in V(G)$. 
So taking $R=\sqrt{7.5}\Delta(G)^{3/2}$ in Statement 2 in Theorem \ref{hyper_fixed_rank} immediately gives the following corollary.

\begin{cor} \label{cor_strong_star_linear}
For every graph $G$ and every $\lceil\sqrt{30}\Delta(G)^{3/2}+\Delta(G)+\frac{1}{3}\rceil$-assignment $L$ of $G$, there exist at least $(7.5\Delta(G)^3)^{|V(G)|/2}$ proper $L$-colorings $\vph$ such that every 4-vertex path in $G$ has a color used exactly once, and for every vertex $v$ with $\deg_G(v) \geq 3$ and for any three neighbors of $v$, some color is used exactly once on those three neighbors.
In particular, every component of the subgraph of $G$ induced by any two arbitrarily chosen color classes of $\vph$ is a path on at most three vertices.
\end{cor}

The coloring satisfying the conclusion of Corollary~\ref{cor_strong_star_linear} is both a \EmphE{star coloring}{0mm}, which is a proper coloring with no bi-colored 4-vertex path, and also a
\EmphE{linear coloring}{0mm}, which is a proper coloring such that any two color classes induce a subgraph whose every component is a path.
Yuster \cite{Y} proved that there exist infinitely many graphs $G$ having no linear coloring with at most $\Delta(G)^{3/2}/11$ colors.
So Corollary~\ref{cor_strong_star_linear}, and hence Theorem~\ref{hyper_fixed_rank}, are optimal up to a constant factor.
Corollary~\ref{cor_strong_star_linear} also improves the currently best known upper bound $\max\{50\Delta(G)^{4/3}, 10\Delta(G)^{3/2}\}$ for linear coloring \cite{Y}.
For star coloring, our $(\sqrt{30}+o(1))\Delta(G)^{3/2}$ upper bound is not better than the currently best known upper bounds $2\sqrt{2}\Delta(G)^{3/2}+\Delta(G)$ \cite{EP,WW}, even though both results are only a $O(\sqrt{\log\Delta(G)})$ factor away from the known lower bound $\Omega(\Delta(G)^{3/2}/\sqrt{\log\Delta(G)})$ \cite{FRR}.
However, the coloring we obtain in Corollary~\ref{cor_strong_star_linear} is stronger than a star coloring and than a linear coloring.  
Every bi-chromatic component of the coloring in Corollary~\ref{cor_strong_star_linear} has at most three vertices, while the bi-chromatic components of a star coloring or a linear coloring can have arbitrarily many vertices.

Our next result shows that the fractional version of Conjecture
\ref{max_deg_conj} holds asymptotically.  It follows from a more general
setting for coloring pairs $(G,\Hy)$.

Let $[k]$\aside{$[k]$} denote $\{1,\ldots,k\}$, for each $k\in\Z^+$.
Let $a$ and $b$ be positive integers.
An \EmphE{$(a:b)$-coloring}{4mm} of a graph or a hypergraph assigns to each
vertex a $b$-element subset of $[a]$.  An $(a:b)$-coloring $\vph$ of a graph
$G$ is \EmphE{proper}{0mm} if for every $j \in [a]$, the preimage
$\varphi^{-1}(j)$ is a stable set in $G$.  An $(a:b)$-coloring $\vph$ of a
hypergraph $\Hy$ is \EmphE{conflict-free}{0mm} if for every edge $e$ of $\Hy$,
there exist at least $b$ elements $\ell$ of $[a]$ such that $|e \cap
\vph^{-1}(\ell)|=1$.  Let $G$ be a graph and $\Hy$ be a hypergraph with
$V(\Hy)=V(G)$.  An $(a:b)$-coloring $\vph$ of $(G,\Hy)$ is \EmphE{fractionally
proper conflict-free}{-5mm} if it is a proper $(a:b)$-coloring of $G$ and a
conflict-free $(a:b)$-coloring of $\Hy$.  For a positive real number $t$,
$(G,\Hy)$ is \EmphE{fractionally properly conflict-free $t$-colorable}{-1mm} if
there exists a proper conflict-free $(x:y)$-coloring of $(G,\Hy)$ for some
positive integers $x,y$ with $x/y \leq t$.  

Fractional proper conflict-free coloring for $(G,\Hy)$, defined above, is a
natural linear programming relaxation of proper conflict-free coloring for
$(G,\Hy)$.  We prove that $\Delta(\Hy)$ is no longer required for upper
bounding the number of colors, if we consider fractional coloring and
$\rank(\Hy) \leq \Delta(G)$.

\begin{theorem} \label{fractional_intro}
For every $\epsilon>0$, there exists $d_0$ such that if $\Delta \geq d_0$ and
$G$ is a graph with maximum degree at most $\Delta$, then $(G,\Hy)$ is
fractionally properly conflict-free $(1+\epsilon)\Delta$-colorable for every
hypergraph $\Hy$ with $V(\Hy)=V(G)$ and $\rank(\Hy) \leq \Delta$.
\end{theorem}

Theorem \ref{fractional_intro} is asymptotically optimal since, for each $\Delta\ge 2$, there are
infinitely many connected graphs with maximum degree $\Delta$ that are not properly $(\Delta-1)$-colorable.

For a graph $G$, we say that an $(a:b)$-coloring of $G$ is a \EmphE{fractional
proper conflict-free coloring of $G$}{0mm} if it is a proper $(a:b)$-coloring
of $G$ such that for every non-isolated vertex $v$ of $G$, at least $b$ colors
appear exactly once on $N(v)$.  If we take $\Hy$ so that $E(\Hy)$ contains all
nonempty subsets of $V(G)$ with size at most $\Delta$, then
Theorem~\ref{fractional_intro} leads to the following corollary.

\begin{cor} \label{fractional_cor}
For every $\epsilon>0$, there exists $d_0$ such that if $\Delta \geq d_0$ and
$G$ is a graph with maximum degree at most $\Delta$, then there exists a proper
$(a:b)$-coloring $\vph$ of $G$ for some positive integers $a$ and $b$ with $a
\leq (1+\epsilon)\Delta b$ such that
\begin{enumerate}
	\item $\vph$ is a fractional proper conflict-free $(a:b)$-coloring of $G$,
	\item for any set $C$ of colors with size less than $5b/2$, every component of the subgraph of $G$ induced by the vertices that use only colors in $C$ has at most two vertices, and
	\item $|\vph(v) \cap \vph(w)| \leq b/2$ for any distinct vertices $v,w$ of $G$.
\end{enumerate}
\end{cor}

Statement 1 of Corollary \ref{fractional_cor} implies that the
fractional version of Conjecture \ref{max_deg_conj} holds asymptotically.
Statement 2 of Corollary \ref{fractional_cor} gives the asymptotically optimal
upper bound for the fractional version of many types of coloring that address
properties of bi-chromatic components, such as acyclic coloring, star coloring,
frugal coloring, and linear coloring.  It is also a fractional analogue of
Corollary~\ref{cor_strong_star_linear}, but now each $2b$-colored component has
at most two vertices, which is clearly optimal.

The paper is organized as follows.
To emphasize the key ideas in our proofs of Theorems~\ref{main_graph}, \ref{hyper_main} and \ref{hyper_fixed_rank}, for clarity we first present the proofs assuming certain estimates of quantities involving 2-associated Stirling numbers, which count the number of partitions of a set with certain properties.  
In Section~\ref{sec2} (and the appendix) we prove those estimates via a sequence of lemmas.  
Finally, we prove Theorem~\ref{fractional_intro} and Corollary~\ref{fractional_cor} in Section~\ref{frac-sec}.

\section{Proof of Theorems~\ref{main_graph}, \ref{hyper_main}, and \ref{hyper_fixed_rank}}

In this section, we prove Theorems~\ref{main_graph}, \ref{hyper_main}, and \ref{hyper_fixed_rank}.
The proofs use a clever inductive counting argument introduced by
Rosenfeld~\cite{rosenfeld1} and extended by Wanless and Wood~\cite{WW}.
 This technique works well for many problems amenable to the Lov\'{a}sz Local
Lemma or entropy compression, but often gives~simpler proofs.
Our proof actually works for a slightly more general setting.
For an integer $t$, a graph $G$, and a hypergraph $\Hy$ with $V(\Hy)=V(G)$, we say that $\vph$ is a \EmphE{proper $t$-conflict-free coloring of $(G,\Hy)$}{-9mm} if it is a proper coloring of $G$ such that for every $f \in E(\Hy)$,\footnote{For clarity, we always use $e$ to denote the base of the natural logarithm, which is the constant $2.71828\cdots$. For an edge of a hypergraph, we typically use $f$.} there exists a color that is used $k$ times by $\vph$ on $f$ for some $k \in [t]$.
Note that conflict-free colorings of $(G,\Hy)$ are exactly $1$-conflict-free colorings of $(G,\Hy)$.

When applying the aforementioned inductive counting argument to proper
$t$-conflict-free coloring, the computation involves $t$-associated Stirling
numbers of the second kind, denoted by $S_t(d,i)$; for positive integers $t$,
$i$, and $d$, the quantity \Emph{$S_t(d,i)$} is defined as the number of
partitions of the set $[d]$ into $i$ parts, each of size at least $t$.
Now we can state and prove our first key lemma.

\begin{lem} \label{rosenfeld}
Let $G$ be a graph and $\Hy$ be a hypergraph with $V(G)=V(\Hy)$. 
Let $t$ be a positive integer.  Let $\beta$ be a real number. 
If $a$ is a real number such that 
	\begin{align}
		a \geq \Delta(G)+\beta+\sum_{f \in E(\Hy), f \ni v}\sum_{i=1}^{\lfloor |f|/(t+1) \rfloor} S_{t+1}(|f|,i) \cdot \beta^{i-|f|+1}
	\label{main-hyp}
	\end{align}
for every $v \in V(G)$, then for every $a$-assignment $L$ of $G$, there are at least $\beta^{|V(G)|}$ proper $t$-conflict-free $L$-colorings of $(G,\Hy)$.
\end{lem}

\begin{proof}
For a subset $Z$ of $V(G)$, an $L$-coloring $\vph$ of $G[Z]$ is a \emph{proper
$t$-conflict-free partial coloring on $Z$} if 
	\begin{itemize}
		\item[(a)] $\vph$ is a proper coloring of $G[Z]$, and 
		\item[(b)] for each $f\in E(\Hy)$, if $f\subseteq Z$, then
$\vph$ uses some color exactly $k$ times on $f$ for some $k \in [t]$.
	\end{itemize}
For each subset $Z$ of $V(G)$, denote by \Emph{$R(Z)$} the number of proper
$t$-conflict-free partial $L$-colorings on $Z$.  For every nonempty subset $Z$
of $V(G)$, and every $v\in Z$, we will prove by induction on $|Z|$:
\begin{align}
	R(Z)&\ge \beta \cdot R(Z \setminus \{v\}). \label{key-ineq}
\end{align}
For each $v\in V(G)$, we have $R(\{v\})=a\ge \beta$.  
So the desired base case holds. 
Thus, induction on $|Z|$ will give $R(Z)\ge \beta^{|Z|}$ for all $Z$.  
In particular, $R(V(G))\ge \beta^{|V(G)|}$.

Fix some $Z \subseteq V(G)$ and $v \in Z$.
If we extend a proper $t$-conflict-free partial coloring on $Z \setminus \{v\}$
by further coloring $v$ with a color that is not used on its colored neighbors,
then we form a proper $L$-coloring of $G[Z]$.  Note that there are at least
$a-\Delta(G)$ ways to extend a proper $t$-conflict-free partial coloring on
$Z \setminus \{v\}$ to a proper coloring of $G[Z]$.  Hence there are at least
$(a-\Delta(G)) \cdot R(Z\setminus\{v\})$ $L$-colorings of $G[Z]$ satisfying (a).
We show that at least $\beta\cdot R(Z\setminus\{v\})$ of these also satisfy (b).

A proper $L$-coloring $\vph$ of $G[Z]$ is \Emph{bad} if it is not a proper
$t$-conflict-free partial coloring on $Z$ but its restriction to $Z \setminus
\{v\}$ is a proper $t$-conflict-free partial coloring on $Z \setminus\{v\}$.
Let \Emph{\mbox{$B$}} be the set of bad $L$-colorings.  
For every $f \in E(H)$ with $v \in f$ and $f \subseteq Z$, let
\Emph{$B_f$}:= $\{\vph \in B: \vph$ does not use any color exactly $k$
times on $f$ for every $k\in[t]\}$; these are the colorings that are bad for $f$.
So $B = \bigcup_{f \in E(\Hy), f \ni v}B_f$.
To prove~\eqref{key-ineq}, it suffices to show $(a-\Delta(G)) \cdot
R(Z\setminus\{v\})-|B| \geq \beta \cdot R(Z \setminus \{v\})$.

We claim that for every edge $f$ of $\Hy$ containing $v$,
	\begin{align}
	    |B_f| \leq \sum_{i=1}^{\lfloor |f|/(t+1) \rfloor} S_{t+1}(|f|,i) \cdot R(Z\setminus\{v\})\beta^{i-|f|+1}. \label{Bf-ineq}
	\end{align}
This claim implies this lemma since the number of proper $t$-conflict-free partial colorings of $G[Z]$ is at least
	\begin{align*}
		& R(Z\setminus\{v\})(a-\Delta(G)) - \sum_{f \in E(\Hy), f \ni v} |B_f| \\
		\ge~& R(Z\setminus\{v\})\left[a-\Delta(G) - \sum_{f \in E(\Hy), f \ni v}\sum_{i=1}^{\lfloor |f|/(t+1) \rfloor} S_{t+1}(|f|,i) \cdot \beta^{i-|f|+1}\right]\\
		\ge~& \beta R(Z\setminus \{v\}).
	\end{align*}
Here the final inequality follows from~\eqref{main-hyp}.

Now we prove \eqref{Bf-ineq}.
Fix an edge $f$ of $\Hy$ containing $v$.
For every $\vph \in B_f$, let \Emph{\mbox{$\PP_\vph$}} be the partition $\{\vph^{-1}(j) \cap f: j \in {\mathbb Z}^+\}$ of $f$. 
These are the color classes of $\vph|_f$, which is the coloring obtained by restricting $\vph$ to $f$.
Since $\vph \in B_f$, every part in $\PP_\vph$ has size at least $t+1$.
Hence the number of possibilities for $\PP_{\vph}$ is at most $S_{t+1}(|f|,i_\vph)$, where $i_\vph$ is the number of colors used in $\vph|_f$.
Note that for every partition $\PP$ of $f$ into $i$ parts each having size at least $t+1 \geq 2$, there exists a subset \Emph{$T_\PP$} of $f$ consisting of a vertex in each part of $\PP$ such that $v \not \in T_\PP$.
We define \Emph{$T_\vph$} to be $T_{\PP_\vph} \cup (Z-f)$.
Note that $\vph$ is uniquely determined by $\PP_\vph$ and $\vph|_{T_\vph}$.
Moreover, since $T_\vph \subseteq Z \setminus \{v\}$, the partial coloring $\vph|_{T_\vph}$ is a proper $t$-conflict-free partial $L$-coloring on $T_\vph = Z \setminus (f \setminus T_{\PP_\vph})$.
Hence, for every partition $\PP$ of $f$ into parts each having size at least $t+1$, the number of possibilities for $\vph|_{T_\vph}$, among all colorings $\vph$ in $B_f$ with $\PP_\vph=\PP$, is at most $R(Z \setminus (f \setminus T_{\PP}))$.
By applying the inductive hypothesis $|f|-|\PP|-1$ times, we see that 
	\begin{align*}
		R(Z\setminus\{v\}) \ge R(Z\setminus (f \setminus T_\PP))\beta^{|f|-|\PP|-1}.
	\end{align*}
Equivalently, $R(Z \setminus (f \setminus T_\PP)) \leq R(Z\setminus\{v\})\beta^{|\PP|-|f|+1}.$
Therefore, for any fixed integer $i$, the number of colorings $\vph$ in $B_f$ with $|\PP_\vph|=i$ is at most $S_{t+1}(|f|,i) \cdot R(Z\setminus\{v\})\beta^{i-|f|+1}$.
Since every color is used at least $t+1$ times or zero times on $f$, we have $1 \leq i \leq \lfloor |f|/(t+1) \rfloor$.

This proves the lemma.
\end{proof}

As we have seen in Lemma \ref{rosenfeld}, estimates for $t$-associated Stirling
numbers of the second kind are crucial.  In this paper, we will only need
sophisticated estimates for the case $t=2$.  These $S_2(d,i)$ have been widely
studied in the literature (see~\cite{Bona-Mezo,OEIS} and the references
therein), but the known formulas are unhelpful for our purposes here.  

Instead, we will prove the following two estimates for $S_2(d,i)$.
We prove the second in Section~\ref{sec2}. But the proof of the first essentially 
consists of straightforward but tedious calculation, so we defer it to the
appendix.

\begin{lem} \label{lem:clm1_simple}
Let $d$ and $R$ be positive integers and $\beta$ be a real number with $\beta \geq 0.6R \geq 14$.
If $3\le {d}\le \beta^{19/20}$ and $\beta \geq 600$ and $R \geq 750$, then 
$$
\sum_{i=1}^{\left\lfloor {d}/2\right\rfloor}S_2(d,i)\beta^{i-d+1}\le  R^{-1/2}.
$$
\end{lem}

\begin{lem} \label{upper-lem_simple}
Let $d$ and $R$ be positive integers with $110 \leq d \leq R$. 
If $\epsilon$, $c$, and $\beta$ are real numbers such that $0.6251\le \epsilon<1$, $0.3\le c<\frac{\epsilon}{2}$ 
and $\epsilon R \leq \beta \leq R$, then 
$$\sum_{i=1}^{\lfloor d/2\rfloor} S_2(d,i)\beta^{i-d+1} \leq \frac{d\beta}2\left(\frac{2c}{\epsilon}\right)^{(1-c)d} + 2R^3 \left(\frac{\epsilon}{0.6251}\right)^{-d/2}.$$ 
Moreover, if $d\ge \max\{\frac{7.6\log R}{\log(\epsilon/0.6251)},\frac{2.5\log R}{(1-c)\log(\epsilon/(2c))}\}$, then $$\sum_{i=1}^{\lfloor d/2\rfloor} S_2(d,i)\beta^{i-d+1} \leq R^{-1/2}.$$
\end{lem}

Now we can prove a special case of Theorem \ref{main_graph} for graphs with
minimum degree at least three.

\begin{lem} \label{mindeg3}
Fix a positive integer $\Delta\ge 6.5 \cdot 10^7$, fix a real number
$\beta$ with $\Delta \geq \beta \ge 0.6550826\Delta$, and let
$a:=\left\lceil\Delta+\beta+\sqrt{\Delta}\right\rceil$.  If $G$ is a graph with
minimum degree at least 3 and maximum degree at most $\Delta$, and $L$ is an
$a$-assignment for $G$, then there are at least $\beta^{|V(G)|}$ proper
conflict-free $L$-colorings of $G$.  Analogous statements hold when $\Delta\ge
4000$ and $\beta\ge 0.\overline{6}\Delta$ and when $\Delta\ge 750$ and
$\beta\ge 0.8\Delta$.
\end{lem}

\begin{proof}
Let \Emph{\mbox{$\Hy$}} be the hypergraph with $V(\Hy)=V(G)$ and $E(\Hy) =\{N(v): v \in V(G)\}$.  
By Lemma~\ref{rosenfeld}, it suffices to
show $\sum_{f \in E(\Hy), f \ni v}\sum_{i=1}^{\lfloor |f|/2 \rfloor}S_2(|f|,i)
\cdot \beta^{i-|f|+1} \leq \sqrt{\Delta}$ for every $v \in V(G)$.  Note that
$\Delta(\Hy) \leq \Delta$.  So it suffices to show, for every $f \in E(\Hy)$, that
\begin{align}
\sum_{i=1}^{\lfloor |f|/2
\rfloor}S_2(|f|,i) \cdot \beta^{i-|f|+1} \leq 1/\sqrt{\Delta}.
\label{mindeg3-eq}
\end{align}

Fix an edge $f$ of $\Hy$.
Every vertex of $G$ has degree at least 3 and at most $\Delta$, so $3 \leq |f| \leq \Delta$. 
Our assumptions for $\Delta$ and $\beta$ imply $\beta \geq 600$.
So \eqref{mindeg3-eq} holds when $|f| \leq \beta^{19/20}$ by Lemma \ref{lem:clm1_simple}.
Hence we may assume $|f| > \beta^{19/20}$.
By Lemma \ref{upper-lem_simple}, it suffices to show $\beta^{19/20} \geq \max\{\frac{7.6\log \Delta}{\log(\epsilon/0.6251)},\frac{2.5\log \Delta}{(1-c)\log(\epsilon/(2c))}\}$, for corresponding choices of $\epsilon$ and $c$.

Let $\epsilon=0.6550826$ and $c=0.32754$.
So $\Delta \geq \beta \geq \epsilon \Delta$.
We have 
$\max\{\frac{7.6}{\log (\epsilon/0.6251)}, \frac{2.5}{(1-c) \log (\epsilon/(2c))}\}\le \max\{164,936689\}=936689$.
By considering the derivative, we know $(\epsilon \Delta)^{19/20}/\log \Delta$ is increasing when $\Delta>10$.
If $\Delta \geq 6.5 \cdot 10^7$, then $\frac{\beta^{19/20}}{\log\Delta}
\geq \frac{(\epsilon \Delta)^{19/20}}{\log \Delta} \geq \frac{(\epsilon \cdot
6.5 \cdot 10^7)^{19/20}}{\log(6.5 \cdot 10^7)} \geq 983377 > 936689 \geq
\max\{\frac{7.6}{\log (\epsilon/0.6251)}, \frac{2.5}{(1-c) \log (\epsilon/(2c))}\}$.
So we are done.

Now we assume $\epsilon=0.\overline{6}$ and $c=0.3272$.
We have 
$\max\{\frac{7.6}{\log (\epsilon/0.6251)}, \frac{2.5}{(1-c) \log (\epsilon/(2c))}\}\le \max\{120, \allowbreak 201\} =201$.
If $\Delta \geq 4000$, then $\frac{\beta^{19/20}}{\log\Delta} \geq
\frac{(\epsilon \Delta)^{19/20}}{\log \Delta} \geq \frac{(\epsilon \cdot 4000)^{19/20}}{\log 4000} \geq 216 \geq 201$.
So we are done.

Finally, we assume $\epsilon=0.8$ and $c=0.32$.  We have 
$\max\{\frac{7.6}{\log (\epsilon/0.6251)}, \frac{2.5}{(1-c) \log (\epsilon/(2c))}\}\le \max\{32, \allowbreak 22\} =32$.
If $\Delta \geq 750$, then $\frac{\beta^{19/20}}{\log\Delta} \geq
\frac{(\epsilon \Delta)^{19/20}}{\log \Delta} \geq \frac{(\epsilon \cdot
750)^{19/20}}{\log 750} \geq 65 \geq 32$.
This proves the lemma.
\end{proof}

Now we are ready to prove Theorem \ref{main_graph}.

\begin{proof}[Proof of Theorem \ref{main_graph}]
Suppose that $G$ is a counterexample with the minimum number of vertices. 
Clearly, $G$ is connected and has at least three vertices.  Let $v$ be a vertex
of $G$ with smallest degree.  By Lemma \ref{mindeg3}, the degree of $v$ is 1 or
2.  Let $x$ and $y$ be the neighbors of $v$, where $x=y$ if $v$ has degree 1. 
If $x \neq y$, and $x$ and $y$ are non-adjacent, then let $G':=G-v+xy$;
otherwise, let $G':=G-v$.  Let $L'$ be the restriction of $L$ to $V(G')$.
Since $G'$ has maximum degree at most $\Delta$, the minimality of $G$ implies
that there exist at least $\beta^{|V(G)|-1}$ proper conflict-free
$L'$-colorings of $G'$.  Hence, to obtain a contradiction, it suffices to show
that for every proper conflict-free $L'$-coloring $\vph$ of $G'$, there are at
least $\beta$ ways to extend $\vph$ to a proper conflict-free $L$-coloring of $G$.

Fix a proper conflict-free $L'$-coloring $\vph$ of $G'$.  Since $G$ is
connected and has at least three vertices, $x$ and $y$ each have degree at
least one in $G'$, by the choice of $v$.  So there exist colors $c_x$ and $c_y$
such that $c_x$ appears on $N_{G'}(x)$ exactly once and $c_y$ appears on
$N_{G'}(y)$ exactly once.  If possible, we choose $c_x \neq \vph(y)$ and $c_y
\neq \vph(x)$.  If $G'=G-v$, then either $N_G(v)=\{x\}$, or $xy \in E(G')$ and
$\vph(x) \neq \vph(y)$, so there are at least $a-4 \geq \beta$ ways to extend
$\vph$ to a proper conflict-free coloring of $G$ by coloring $v$ with a color
in $L(v) \setminus \{\vph(x),\vph(y),c_x,c_y\}$, a contradiction.  So $G'=G-v+xy$
and $x \neq y$.  For each $u \in \{x,y\}$ and $u' \in \{x,y\} \setminus \{u\}$,
if $c_u = \vph(u')$, then let $S_u := \{\vph(z): z \in N_G(u) \setminus
\{v\}\}$; otherwise, let $S_u:=\{c_u\}$.  For the former, since $c_u$ is chosen
to be different from $\vph(u')$ if possible, we know that every color appears
on $N_{G}(u) \setminus \{v\}$ zero times or at least twice, so $|S_u| \leq
|N_{G}(u) \setminus \{v\}|/2 \leq (\Delta-1)/2$; for the latter, $|S_u|=1 \leq
(\Delta-1)/2$.  But there are at least $a-|S_x|-|S_y|-2 \geq a-(\Delta-1)-2
\geq \beta$ ways to extend $\vph$ to a proper conflict-free coloring of $G$ by
coloring $v$ with a color in $L(v) \setminus (S_x \cup S_y \cup
\{\vph(x),\vph(y)\})$, a contradiction.
\end{proof}

To prove Theorem \ref{hyper_main}, we use the following estimate for $S_2(d,i)$,
which we will prove in Section~\ref{sec2}.

\begin{lem} \label{S_2_simple_bound}
If $d$ is a positive integer and $\beta$ is a real number with $d < \beta$,
then 
$$
	\sum_{i=1}^{\left\lfloor {d}/2\right\rfloor}S_2(d,i)\beta^{i-d+1}\le \beta \cdot \frac{(d/\beta)^{\lceil d/2 \rceil}}{1-\frac{d}{\beta}}.
$$
\end{lem}

Now we prove Theorem \ref{hyper_main}.

\begin{proof}[Proof of Theorem \ref{hyper_main}]
By Lemma \ref{rosenfeld}, it suffices to prove that $\sum_{f \in E(\Hy), f \ni v}\sum_{i=1}^{\lfloor |f|/2 \rfloor}S_{2}(|f|,i) \cdot \beta^{i-|f|+1} \leq \deg_\Hy(v) \cdot \max\{2\beta(\frac{(\log R)^{2}}{\beta})^{\lceil \mr_\Hy(v)/2 \rceil},(1-10^{-8})^{(\log R)^{2}}\}$ for every $v \in V(G)$, when $R \geq e^{3.1 \cdot 10^8}$.
And it suffices to prove $\sum_{i=1}^{\lfloor |f|/2 \rfloor}S_{2}(|f|,i) \cdot \beta^{i-|f|+1} \leq \max\{2\beta(\frac{(\log R)^{2}}{\beta})^{\lceil \mr_\Hy(v)/2 \rceil},(1-10^{-8})^{(\log R)^{2}}\}$ for every $v \in V(G)$ and $f \in E(\Hy)$ containing $v$, when $R \geq e^{3.1 \cdot 10^8}$.

Fix a vertex $v$ of $G$ and an edge $f$ of $\Hy$ containing $v$.

Suppose $|f| \leq (\log R)^2$.
Since $R \geq e^{5 \cdot 10^6}$, we get $|f| \leq (\log R)^2 \leq 0.5 \cdot 0.65R < \beta$.
By Lemma~\ref{S_2_simple_bound}, 
	$$\sum_{i=1}^{\lfloor |f|/2\rfloor} S_2(|f|,i)\beta^{i-|f|+1} \leq \beta\frac{(|f|/\beta)^{\lceil |f|/2 \rceil}}{1-\frac{|f|}{\beta}} \leq \beta\frac{((\log R)^{2}/\beta)^{\lceil \mr_\Hy(v)/2 \rceil}}{1-\frac{(\log R)^{2}}{0.65R}} \leq 2\beta^{1-\lceil \frac{\mr_\Hy(v)}{2} \rceil}(\log R)^{2\lceil \frac{\mr_\Hy(v)}{2} \rceil}.$$

Assume instead that $|f| > (\log R)^2$.
Let $\epsilon=0.6550826$ and $c=0.32754125$.
Now $0.6251 \leq \epsilon < 1$, $0.3 \leq c < \epsilon/2$, and $\epsilon R \leq \beta \leq R$.
Since $|f| \leq \rank(\Hy) \leq R$ and $R \geq |f| \ge 110$, by Lemma \ref{upper-lem_simple}, 
	\begin{align*}
		\sum_{i=1}^{\lfloor |f|/2\rfloor} S_2(|f|,i)\beta^{i-|f|+1} & \leq 
\frac{|f|R}2\left(\frac{2\cdot 0.32754125}{0.6550826}\right)^{(1-0.32754125)|f|}+ 2R^3\left(\frac{0.6550826}{0.6251}\right)^{-|f|/2} \\ 
		& \le 2.5R^3\cdot (1-10^{-7})^{|f|}. 
	\end{align*}
Because $|f| \geq (\log R)^{2}$, 
	\begin{align*}
		2.5R^3 \cdot (1-10^{-7})^{|f|} & \leq 2R^2 \cdot (1-10^{-7})^{(\log R)^{2}} \\
		& = 2.5R^3 \cdot (\frac{1-10^{-7}}{1-10^{-8}})^{(\log R)^{2}} \cdot (1-10^{-8})^{(\log R)^{2}} \\
		& \leq 2.5R^3 \cdot (\frac{1}{1+10^{-8}})^{(\log R)^{2}} \cdot (1-10^{-8})^{(\log R)^{2}}. 
	\end{align*}
Note that $\log(1+10^{-8}) \geq 9.99999995 \cdot 10^{-9}$.
	Since $R \geq e^{3.1 \cdot 10^8}$, we have $2.5R^3 \leq R^{3.01}$ and $\log R \geq 3.1 \cdot 10^8$, so $3.01\log R \leq (\log R)^{2} \cdot 9.99999995 \cdot 10^{-9} \leq (\log R)^{2} \cdot \log(1+10^{-8})$.
Hence $2.5R^3 \cdot (\frac{1}{1+10^{-8}})^{(\log R)^{2}} \cdot (1-10^{-8})^{(\log R)^{2}} \leq (1-10^{-8})^{(\log R)^{2}}$.
This proves the theorem.
\end{proof}

Finally, we prove Theorem \ref{hyper_fixed_rank}.

\begin{proof}[Proof of Theorem \ref{hyper_fixed_rank}]
Let $v \in V(G)$ and $f \in E(\Hy)$ with $v \in f$.
Since $\rank(\Hy) \leq r$, $|f| \leq r$.
If $R \geq (1+\frac{1}{\epsilon})r$, then $1-\frac{r}{R} \geq \frac{1}{1+\epsilon}$, so by Lemma~\ref{S_2_simple_bound}, $$\sum_{i=1}^{\lfloor |f|/2\rfloor} S_2(|f|,i)R^{i-|f|+1} \leq R\frac{(|f|/R)^{\lceil |f|/2 \rceil}}{1-\frac{|f|}{R}} \leq R\frac{(r/R)^{\lceil \mr_\Hy(v)/2 \rceil}}{1-\frac{r}{R}} \leq (1+\epsilon)R^{1-\lceil \frac{\mr_\Hy(v)}{2} \rceil}r^{\lceil \frac{\mr_\Hy(v)}{2} \rceil}.$$
Hence $\sum_{f \in E(\Hy), f \ni v}\sum_{i=1}^{\lfloor |f|/2 \rfloor}S_{2}(|f|,i) \cdot R^{i-|f|+1} \leq \deg_\Hy(v) \cdot (1+\epsilon)R^{1-\lceil \frac{\mr_\Hy(v)}{2} \rceil}r^{\lceil \frac{\mr_\Hy(v)}{2} \rceil}$.
Therefore, Statement 1 of this theorem follows from Lemma \ref{rosenfeld} (with taking $\beta=R$).

Now we assume $r \leq 4$.
Then $3 \leq |f| \leq 4$. 
If $|f|=3$, then $$\sum_{i=1}^{\lfloor |f|/2\rfloor} S_2(|f|,i)R^{i-|f|+1} = S_2(3,1)R^{-1} = R^{-1}.$$ 
If $|f|=4$, then $$\sum_{i=1}^{\lfloor |f|/2\rfloor} S_2(|f|,i)R^{i-|f|+1} = S_2(4,1)R^{-2} + S_2(4,2)R^{-1} = R^{-2} + \frac{1}{2}{4 \choose 2}R^{-1} = R^{-2}+3R^{-1}.$$
Hence $\sum_{f \in E(\Hy), f \ni v}\sum_{i=1}^{\lfloor |f|/2 \rfloor}S_{2}(|f|,i) \cdot R^{i-|f|+1} \leq \deg_\Hy(v) \cdot (3R^{-1}+R^{-2})$.
Therefore, Statement 2 of this theorem follows from Lemma \ref{rosenfeld} (with taking $\beta=R$).
\end{proof}

\section{Estimates for 2-Associated Stirling Numbers}
\label{sec2}

Recall that, to prove our results in the previous section,
we assumed the correctness of estimates about summations 
involving $S_2(d,i)$ (Lemmas~\ref{lem:clm1_simple}, \ref{upper-lem_simple}, and
\ref{S_2_simple_bound}).  In this section, we prove the second and third of these; 
the proof of Lemma~\ref{lem:clm1_simple} consists of calculation that is 
tedious but straightforward, so we defer it to the appendix.
For simplicity, we denote $S_2(d,i)$ by \Emph{$E_i(d)$} for any positive integers $d$ and $i$.
We will extensively use two simple upper bounds for $E_i(d)$, combined via
$\min\{\cdot,\cdot\}$ in the following lemma.

\begin{lem} \label{two_basic_bound}
For positive integers $i$ and $d$,
	\begin{align*}
		E_i(d) \leq & \min\Big\{{{d}\choose i}i^{{d}-i}2^{-i}, \\
		& \frac{d!}{\lfloor d/2 \rfloor!2^{d/2}} \cdot {\bf 1}_{d/2 \in {\mathbb Z}} + \sum_{j=\max\{3i-d,0\}}^{i-1}\binom{d}{2j}\frac{(2j)!}{j!2^j} \binom{d-2j}{3(i-j)} \frac{(3(i-j))!}{(i-j)!(3!)^{i-j}} (i-j)^{d-2j-3(i-j)}\Big\},
	\end{align*}
where ${\bf 1}_{d/2 \in {\mathbb Z}}=1$ if $d/2 \in {\mathbb Z}$, and ${\bf 1}_{d/2 \in {\mathbb Z}}=0$ otherwise.
\end{lem}

\begin{proof}
To form a partition of $[d]$ into $i$ parts, each of size at least 2, we can choose $i$ elements to be ``leaders'' and then assign each other element to the part of a leader. (This also forms partitions with parts of size 1 but, since we
seek an upper bound, this is not a problem.)
Since each part has at least two elements, this process overcounts by a factor
of at least $2^i$.  Thus, 
\begin{align*}
E_i(d)\le {{d}\choose i}i^{{d}-i}2^{-i}.
\end{align*}

Now we prove the other upper bound.
To form a partition of $[d]$ into $i$ parts each of size at least 2, we first consider the parts of size exactly 2; denote the number of them by $j$.
To form these $j$ parts of size 2, we choose $2j$ elements and pair them up.
The number of ways to pair $2j$ elements is precisely $\frac{(2j)!}{j!2^j}$.  
If $i=j$, then $d=2i$; if $j<i$, then each of the remaining $i-j$ parts has size at least 3, and to form these, we choose $3(i-j)$ of the remaining elements, group them into triples, then assign each element still remaining to one of these triples.
The number of ways to group $3(i-j)$ elements into triples is $\frac{(3(i-j))!}{(i-j)!(3!)^{i-j}}$.
This gives 
\begin{align*}
E_i(d)\le 
	\frac{d!}{\lfloor d/2 \rfloor!2^{d/2}} \cdot {\bf 1}_{d/2 \in {\mathbb Z}} + \sum_{j=\max\{3i-d,0\}}^{i-1}\binom{d}{2j}\frac{(2j)!}{j!2^j} \binom{d-2j}{3(i-j)} \frac{(3(i-j))!}{(i-j)!(3!)^{i-j}} (i-j)^{d-2j-3(i-j)}. 
\end{align*}
(For the lower bound on the index $j$, note that $2j+3(i-j)\le d$, which implies that $j\ge 3i-d$.)
\end{proof}

\subsection{Proof of Lemma~\ref{S_2_simple_bound}}

We begin by proving Lemma~\ref{S_2_simple_bound}.  For easy reference, we restate it.
This result will also be used when proving Lemma~\ref{lem:clm1_simple}.

\setcounter{lemA}{12}
\begin{lemA} 
\label{S_2_simple_bound_proof}
If $d$ is a positive integer and $\beta$ is a real number with $d < \beta$, then 
$$
	\sum_{i=1}^{\left\lfloor {d}/2\right\rfloor}E_i(d)\beta^{i-d+1}\le \beta \cdot \frac{(d/\beta)^{\lceil d/2 \rceil}}{1-\frac{d}{\beta}}.
$$
\end{lemA}

\begin{proof}
By the first bound in Lemma~\ref{two_basic_bound},
$$\sum_{i=1}^{\left\lfloor {d}/2\right\rfloor}E_i(d)\beta^{i-d+1} \leq
\sum_{i=1}^{\lfloor d/2 \rfloor} {{d}\choose i}i^{{d}-i}2^{-i}\beta^{i-d+1} =
\beta \sum_{i=1}^{\left\lfloor {d}/2\right\rfloor}{{d}\choose
i}2^{-{d}}\left(\frac{2i}{\beta}\right)^{{d}-i}.$$

Since ${d \choose i} \leq 2^d$ and $d < \beta$, 
	$$\beta \sum_{i=1}^{\left\lfloor {d}/2\right\rfloor}{{d}\choose
i}2^{-{d}}\left(\frac{2i}{\beta}\right)^{{d}-i} \le \beta
\sum_{i=1}^{\left\lfloor {d}/2\right\rfloor}\left(\frac{d}{\beta}\right)^{{d}-i}
\le \beta \sum_{j=\left\lceil
{d}/2\right\rceil}^{\infty}\left(\frac{d}{\beta}\right)^j = \beta \frac{({d}/\beta)^{\left\lceil {d}/2\right\rceil}}{1-\frac{d}{\beta}}.$$
\aftermath
\end{proof}

\subsection{Proof of Lemma~\ref{upper-lem_simple}}

In the rest of this section, we prove Lemma~\ref{upper-lem_simple}. 
To upper bound binomial coefficients when proving Lemma~\ref{referee-big}, 
we need the following two upper bounds; for completeness we include their short~proofs.

\begin{prop}
\label{referee1}
If $k$ and $n$ are integers with $0<k<n$, then
$$
\binom{n}{k}\le \frac{n^n}{k^k(n-k)^{n-k}}.
$$
\end{prop}
\begin{proof}
By the Binomial Theorem, we have
$$
n^n = (k+(n-k))^n = \sum_{i=0}^n\binom{n}{i}k^i(n-k)^{n-i}>\binom{n}{k}k^k(n-k)^{n-k}.
$$
\aftermath
\end{proof}

\begin{prop}
\label{referee2}
If $i$ is a positive integer, then
$$
\frac{(2i)!}{i!2^i}\le 2i\left(\frac{2i}e\right)^i.
$$
\end{prop}

\begin{proof}
It is well-known that $\frac{n^n}{e^{n-1}}\le n!\le \frac{n^{n+1}}{e^{n-1}}$; for example, see~\cite{CL} for a complete proof.  
Direct computation gives
$$
\frac{(2i)!}{i!2^i}\le \left(\frac{(2i)^{2i+1}}{e^{2i-1}}\right)/\left(\frac{i^i}{e^{i-1}}\cdot 2^i\right) = \frac{(2i)^{i+1}}{e^i}.
$$
\aftermath
\end{proof}

To prove Lemma~\ref{upper-lem_simple}, we estimate separately the summations of
lower indexed terms and of higher indexed terms.  (We remark that whenever we
write $\sum_{i=a}^b$, we mean that we sum over all terms with index $i$
satisfying $i \in {\mathbb Z}$ and $a \leq i \leq b$, so $a$ and $b$ are not
necessarily integers.)

\begin{lem} \label{lower-lem}
Let $d$ and $R$\aside{$d$, $R$} be positive integers with $d \leq R$.
Let $\epsilon$, $c$, and $\beta$\aaside{$\epsilon$, $c$, $\beta$}{4mm} be real
numbers such that $0<\epsilon<1$ and
$0<c<\frac{\epsilon}{2}$ and $\epsilon R \leq \beta \leq R$.
    Then $$\sum_{i=1}^{cd} E_i(d)\beta^{i-d+1} \leq \frac{d\beta}2 \cdot\left(\frac{2c}{\epsilon}\right)^{(1-c)d}.$$
    Moreover, if $d \ge \frac{\log (R^{2.5})}{(1-c)\log \frac{\epsilon}{2c}}$, then
$$
\sum_{i=1}^{cd} E_i(d)\beta^{i-d+1} \leq \frac12R^{-0.5}.
$$
\end{lem}

\begin{proof}
    By Lemma~\ref{two_basic_bound}, and the fact that ${d\choose i}\le 2^d$,
    \begin{align*}
        \sum_{i=1}^{cd} E_i(d)\beta^{i-d+1} &\le \sum_{i=1}^{cd} {d \choose i} i^{d-i}2^{-i}\beta^{i-d+1}
         \le \beta\sum_{i=1}^{cd}\left(\frac{2i}{\beta}\right)^{d-i}\\
         &\le \beta\sum_{i=1}^{cd}\left(\frac{2cd}{\beta}\right)^{d-i}
         \le \beta\sum_{i=1}^{cd}\left(\frac{2cR}{\epsilon R}\right)^{d-i}
         \le \frac{d\beta}2\left(\frac{2c}{\epsilon}\right)^{(1-c)d}.
    \end{align*}
If also $d\ge \frac{\log (R^{2.5)}}{(1-c)\log\frac{\epsilon}{2c}}$, then we have
    \begin{align*}
        \sum_{i=1}^{cd} E_i(d)\beta^{i-d+1} \leq \frac{d\beta}2\left(\frac{2c}{\epsilon}\right)^{(1-c)d} \leq \frac{d\beta}2\left(\frac{2c}{\epsilon}\right)^{(\log (R^{2.5}))/\log (\epsilon/(2c))}= \frac{d\beta}2R^{-2.5}\le\frac12R^{-0.5}.
    \end{align*}
   \aftermath 
\end{proof}

Our next lemma allows us to focus on partitions counted by $E_i(d)$ that
have many parts of size at least 3, since it shows that these are at
least $1/4$ of all partitions counted by $E_i(d)$.

\begin{lem} \label{p:jkn-lem}
Let $n,k,j$ be positive integers with $j \leq k$.
Perform $k$ independent draws from $\{1,2,...,n\}$ with uniform probability.
Let $p(j,k,n)$ be the probability that at most $j$ distinct elements are collected.
If $2j\le k$ and $2j \leq n$ and $n\ge 110$, then $p(j,k,n)\le 3/4$.
\end{lem}

\begin{proof} 
Clearly $p(j,k+1,n)\le p(j,k,n)$ since drawing fewer times increases the likelihood of at most $j$ distinct draws.
By the union bound and \Cref{referee1} above, letting $\epsilon:=j/n$, we get
	\begin{align}
    		p(j,k,n)&\le p(j,2j,n) \le \binom{n}{j}\left(\frac{j}{n}\right)^{2j} \le \frac{n^n}{j^j(n-j)^{n-j}}\left(\frac{j}n\right)^{2j} \nonumber\\
		& = \frac{n^{n-2j}j^j}{(n-j)^{n-j}} = \frac{n^{n(1-2\epsilon)}(\epsilon n)^{\epsilon n}}{(n(1-\epsilon))^{n(1-\epsilon)}} = \frac{\epsilon^{\epsilon n}}{(1-\epsilon)^{n(1-\epsilon)}} = \left[\frac{\epsilon^\epsilon}{(1-\epsilon)^{(1-\epsilon)}}\right]^n.  \label{p:jkn-ineq} 
	\end{align}
Note that $\frac{1}{n} \leq \epsilon \leq \frac{1}{2}$.
Let $f$ be the function $f(x) = x\log x - (1-x)\log(1-x)$.
So $p(j,k,n) \leq e^{n \cdot f(\epsilon)}$.
Since $f''(x) = \frac{1-2x}{x(1-x)} > 0$ for $\frac{1}{n} \leq x < \frac{1}{2}$, we have that if $2j<n$, then $\frac{1}{n} \leq \epsilon \leq \frac{(n-1)}{2n}$ and $$p(j,k,n) \leq e^{n \cdot f(\epsilon)} \leq \max\{e^{n \cdot f(\frac{1}{n})},e^{n \cdot f(\frac{n-1}{2n})}\}.$$

Note that by Taylor's approximation, $1-x \leq e^{-x}$ for every real number $x$ with $0<x<1$.
When $\epsilon=1/n$, we know $j=1$, so by \eqref{p:jkn-ineq}, $$e^{n \cdot f(\frac{1}{n})} = \frac{n^{n-2}\cdot 1}{(n-1)^{n-1}}\le \left(\frac{n}{n-1}\right)^{n-1}\cdot \frac1n = \left(1+\frac{1}{n-1}\right)^{n-1}\cdot \frac1n \le \frac en < \frac{3}{4}.$$
When $\epsilon=\frac{n-1}{2n}$, we know $j=\frac{n-1}{2}$, so by \eqref{p:jkn-ineq}, $$e^{n \cdot f(\frac{n-1}{2n})} = \frac{n^1\left(\frac{n-1}2\right)^{(n-1)/2}}{\left(\frac{n+1}2\right)^{(n+1)/2}} = \frac{2n}{n+1}\left(\frac{n-1}{n+1}\right)^{\frac{n-1}{2}} \leq 2\left(1-\frac{2}{n+1}\right)^{\frac{n-1}{2}} \leq 2e^{-\frac{n-1}{n+1}} \leq 2e^{-\frac{109}{111}} < \frac 34.$$
This shows that if $2j <n$, then $p(j,k,n) < \frac{3}{4}$.

So it remains to consider the case $2j=n$.
Let $X$ be a random variable denoting the number of elements of $\{1,2,...,n\}$ that are \emph{not} drawn. 
Since $2j=n$, $p(j,k,n) \leq p(j,2j,2j) = \Pr[X \geq j] = \Pr[X \geq n/2]$.
By linearity of expectation, we have $E[X]=n(1-1/n)^n$.  
By Markov's inequality, $$p(j,k,n) \leq \Pr[X \geq n/2] \leq \frac{E[X]}{n/2} = 2(1-1/n)^n \le \frac{2}{e} < \frac{3}{4}.$$  
\aftermath
\end{proof}

\begin{lem} \label{connection_p_partition}
Let $i$ and $d$ be positive integers with $i \leq d/2$. 
Then for every integer $t$ with $1 \leq t \leq \min\{\frac{d}{2}-i,\frac{i}{2}\}$, the number of partitions of $[d]$ into $i$ parts each having size at least 2 for which there are at most $t$ parts of size at least 3 is at most $p(t,d-2i,i) \cdot E_i(d)$, where the function $p$ is defined in Lemma \ref{p:jkn-lem}.
\end{lem}

\begin{proof}
For every partition $\PP$ of $[d]$ into $i$ parts each having size at least 2, let $Z_\PP = \{\{m_P,m'_P\}: P \in \PP\}$, where $m_P$ and $m_P'$ are the minimum and the second minimum of the part $P$ of $\PP$, respectively. 
Note that each $Z_\PP$ is a pairing of $2i$ elements of $[d]$.
For every pairing $Z$ of $2i$ elements of $[d]$, let $A_Z$ be the set consisting of the partitions $\PP$ of $[d]$ into $i$ parts each having size at least 2 with $Z_\PP=Z$.
Hence $\{A_Z: Z$ is a pairing of $2i$ elements of $[d]\}$ is a partition of the set of partitions of $[d]$ into $i$ parts each having size at least 2.
Note that for every such a pairing $Z$, there exists a bijection $\iota_Z$ from $A_Z$ to $[i]^{d-2i}$ defined by for every $P \in A_Z$ and $j \in [d-2i]$, the $j$-th entry of $\iota_Z(P)$ equals $k$, where $k$ is the integer such that the $j$-th smallest element in $[d]$ not used in $Z$ is contained in the part $P$ of $\PP$ for which $m_P$ is the $k$-th smallest element in $\{m_{P'}: P' \in \PP\}$.
Hence for every integer $t$ with $1 \leq t \leq \min\{\frac{d}{2}-i,\frac{i}{2}\}$, the number of partitions in $A_Z$ having at most $t$ parts of size at least 3 equals $p(t,d-2i,i) \cdot |A_Z|$.
\end{proof}

Our next lemma provides the key step in proving our upper bound on the sum of the higher indexed terms.

\begin{lem}\label{referee-big}
Let $i$ and $d$ be integers with $d\ge 110$.  
If $0.3d \leq i \leq d/2$, then $E_i(d)\le 8i(0.6251d)^{d-i}$.
\end{lem}

\begin{proof}
Among the partitions counted by $E_i(d)$, the fraction of partitions contain at most $\min\{\frac{d}{2}-i,\frac{i}{2}\}$ parts of size at least 3 is $p(\min\{\frac{d}{2}-i,\frac{i}{2}\},d-2i,i)$ by Lemma \ref{connection_p_partition}; by Lemma~\ref{p:jkn-lem}, this fraction is at most $3/4$.  
So it suffices to show that the number of partitions of $[d]$ into $i$ parts each having size at least 2 with more than $\min\{\frac{d}{2}-i,\frac{i}{2}\}$ parts of size at least 3 is at most $2i(0.6251d)^{d-i}$.  
(Hence we may assume $i<\frac{d}{2}$, since $i=\frac{d}{2}$ implies that $\min\{\frac{d}{2}-i,\frac{i}{2}\} = \frac{d}{2}-i = 0$ and it is impossible to have a partition of $[d]$ with $i=\frac{d}{2}$ parts of size at least 2 and more than 0 part of size at least 3.)

To count these partitions, we first draw $2i$ elements that we pair together.  Then, for each of the $d-2i$
remaining elements, we assign it to one of the $i$ parts formed by the pairing.  For each part
of size at least 3, there are at least 3 choices for the initial pair of elements.  Thus,
there are at least $3^j$ ways to construct a partition with $j$ parts of size at least 3 by this process.  
Since we only count those with $j>\min\{\frac{d}{2}-i,\frac{i}{2}\}$, we get the inequality $$\frac{1}{4}E_i(d) \leq \binom{d}{2i}\cdot \frac{(2i)!}{i!2^i}\cdot\frac{i^{d-2i}}{3^{\min\{\frac{d}{2}-i,\frac{i}{2}\}}}.$$

We first assume $\min\{\frac{d}{2}-i,\frac{i}{2}\} = \frac{d}{2}-i$.
Then
	\begin{align*}
		\frac{1}{4}E_i(d) &\le \binom{d}{2i}\cdot \frac{(2i)!}{i!2^i}\cdot\frac{i^{d-2i}}{3^{d/2-i}}\\
		&\le \frac{d^d}{(2i)^{2i}(d-2i)^{d-2i}}\cdot 2i\left(\frac{2i}e\right)^i\cdot \left(\frac{i}{\sqrt{3}}\right)^{d-2i}\\ 
		& = 2i\frac{d^d}{(2ei)^i}\left(\frac{i}{\sqrt{3}(d-2i)}\right)^{d-2i}, 
	\end{align*}
where the second inequality follows from \Cref{referee1} and~\Cref{referee2}.
Letting $\epsilon := i/d$, 
	\begin{align*}
		E_i(d) &\le 4 \cdot 2i \cdot d^{d-i}\left[\frac{d^i i^{d-2i}}{(2ei)^i(\sqrt{3}(d-2i))^{d-2i}}\right] \\ 
       		&= 4 \cdot 2i \cdot d^{d-i}\left[\frac{d^{\epsilon d} {(\epsilon d)}^{d-2\epsilon d}}{(2e\epsilon d)^{\epsilon d}(\sqrt{3}d(1-2\epsilon))^{d-2\epsilon d}}\right] \\ 
	        &= 4 \cdot 2i \cdot d^{d-i}\left[\frac{{\epsilon}^{d(1-2\epsilon)}}{(2e\epsilon )^{\epsilon d}(\sqrt{3}(1-2\epsilon))^{d-2\epsilon d}}\right] \\ 
       		&= 4 \cdot 2i \cdot d^{d-i}\left[\frac{{\epsilon}^{(1-2\epsilon)}}{(2e\epsilon )^{\epsilon}(\sqrt{3}(1-2\epsilon))^{1-2\epsilon}}\right]^{d} \\ 
       		&= 4 \cdot 2i \cdot\left[d\left[\frac{{\epsilon}^{(1-2\epsilon)}}{(2e\epsilon )^{\epsilon}(\sqrt{3}(1-2\epsilon))^{1-2\epsilon}}\right]^{\frac1{1-\epsilon}}\right]^{d-i}. 
	\end{align*}
Let 
\begin{align*}
       f(\epsilon) &:= \left[\frac{{\epsilon}^{(1-2\epsilon)}}{(2e\epsilon )^{\epsilon}(\sqrt{3}(1-2\epsilon))^{1-2\epsilon}}\right]^{\frac1{1-\epsilon}}. \\
\end{align*}
Since $\min\{\frac{d}{2}-i,\frac{i}{2}\} = \frac{d}{2}-i$, we have $\frac{d}{3} \leq i < \frac{d}{2}$, so $\frac{1}{3} \leq \epsilon < \frac{1}{2}$.
With Mathematica we know that $\max_{1/3\le \epsilon\le 1/2}f(\epsilon)\le 0.6251$.
Thus, we conclude that $E_i(d)\le 4 \cdot 2i \cdot (0.6251d)^{d-i}$, as claimed. 

Now we assume $\min\{\frac{d}{2}-i,\frac{i}{2}\} = \frac{i}{2}$.
That is, $0.3d \leq i \leq \frac{d}{3}$.
Then 
	\begin{align*}
		\frac{1}{4}E_i(d) &\le \binom{d}{2i}\cdot \frac{(2i)!}{i!2^i}\cdot\frac{i^{d-2i}}{3^{i/2}}\\
		&\le \frac{d^d}{(2i)^{2i}(d-2i)^{d-2i}}\cdot 2i\left(\frac{2i}{e\sqrt{3}}\right)^i\cdot i^{d-2i}\\ 
		& = 2i\frac{d^d}{(2\sqrt{3}ei)^i}\left(\frac{i}{d-2i}\right)^{d-2i}, 
	\end{align*}
where the second inequality follows from \Cref{referee1} and~\Cref{referee2}.
Letting $\epsilon := i/d$, 
	\begin{align*}
		E_i(d) &\le 4 \cdot 2i \cdot d^{d-i}\left[\frac{d^i i^{d-2i}}{(2\sqrt{3}ei)^i(d-2i)^{d-2i}}\right] \\ 
       		&= 4 \cdot 2i \cdot d^{d-i}\left[\frac{d^{\epsilon d} {(\epsilon d)}^{d-2\epsilon d}}{(2\sqrt{3}e\epsilon d)^{\epsilon d}(d(1-2\epsilon))^{d-2\epsilon d}}\right] \\ 
	        &= 4 \cdot 2i \cdot d^{d-i}\left[\frac{{\epsilon}^{d(1-2\epsilon)}}{(2\sqrt{3}e\epsilon )^{\epsilon d}(1-2\epsilon)^{d-2\epsilon d}}\right] \\ 
       		&= 4 \cdot 2i \cdot d^{d-i}\left[\frac{{\epsilon}^{(1-2\epsilon)}}{(2\sqrt{3}e\epsilon )^{\epsilon}(1-2\epsilon)^{1-2\epsilon}}\right]^{d} \\ 
       		&= 4 \cdot 2i \cdot\left[d\left[\frac{{\epsilon}^{(1-2\epsilon)}}{(2\sqrt{3}e\epsilon )^{\epsilon}(1-2\epsilon)^{1-2\epsilon}}\right]^{\frac1{1-\epsilon}}\right]^{d-i}. 
	\end{align*}
Let 
\begin{align*}
       g(\epsilon) &:= \left[\frac{{\epsilon}^{(1-2\epsilon)}}{(2\sqrt{3}e\epsilon )^{\epsilon}(1-2\epsilon)^{1-2\epsilon}}\right]^{\frac1{1-\epsilon}}. \\
\end{align*}
Since $0.3d \leq i \leq \frac{d}{3}$, so $0.3 \leq \epsilon \leq \frac{1}{3}$.
With Mathematica we know that $\max_{0.3 \leq \epsilon\le 1/3}g(\epsilon) \le 0.57< 0.6251$.
Thus, we conclude that $E_i(d)\le 4 \cdot 2i \cdot (0.6251d)^{d-i}$, as claimed. 
\end{proof}

Now we can finally use the previous three lemmas to bound the sum of higher indexed terms.

\begin{lem} \label{upper-lem}
Let $d$ and $R$\aside{$d$, $R$} be positive integers with $110 \leq d \leq R$. 
If $\epsilon$, $c$, and $\beta$\aaside{$\epsilon$, $c$, $\beta$}{5.5mm} are real numbers such that $0.6251\le \epsilon<1$ and $0.3\le c<\frac{\epsilon}{2}$ and $\epsilon R \leq \beta \leq R$, then $$\sum_{i=cd}^{\lfloor d/2\rfloor} E_i(d)\beta^{i-d+1} \leq 2R^3\left(\frac{\epsilon }{0.6251 }\right)^{-d/2}.$$ 
Moreover, if $d \ge \frac{7.6\log R}{\log (\epsilon/0.6251)}$, then $$\sum_{i=cd}^{\lfloor d/2\rfloor} E_i(d)\beta^{i-d+1}\le \frac12R^{-0.5}.$$
\end{lem}

\begin{proof}
By Lemma~\ref{referee-big}, we have
	\begin{align*}
		\sum_{i=cd}^{d/2}E_i(d)\beta^{i-d+1} & \le \sum_{i=cd}^{d/2}8i(0.6251d)^{d-i}\beta^{i-d+1}\\
		&= 8\beta\sum_{i=cd}^{d/2}i\left(\frac{\beta}{0.6251 d}\right)^{i-d}\\
		&= 8\beta\sum_{i=cd}^{d/2}i\left(\frac{\epsilon R}{0.6251 R}\right)^{i-d}\\
		&\le 2R^3\left(\frac{\epsilon }{0.6251 }\right)^{-d/2}. 
	\end{align*}

If also $d\ge (7.6\log R)/(\log (\epsilon/0.6251))$, then $$\sum_{i=cd}^{d/2}E_i(d)\beta^{i-d+1} \le 2R^3\left(\frac{\epsilon}{0.6251}\right)^{(-3.8\log R)/(\log (\epsilon/0.6251))} =2R^3(R^{-3.8}) = 2R^{-0.8}\le \frac12 R^{-0.5}.$$
\aftermath
\end{proof}

\begin{cor}
Lemma~\ref{upper-lem_simple} is true.
\end{cor}

\begin{proof}
This follows immediately from Lemmas~\ref{lower-lem} and \ref{upper-lem}.
\end{proof}

\section{Fractional Coloring}
\label{frac-sec}

The goal of this section is to prove Theorem~\ref{fractional_intro}, which is an asymptotically optimal theorem for fractional proper conflict-free coloring. 
The following is a restatement.

\begin{theorem} \label{fractional}
For every $\epsilon>0$, there exists $d_0$ such that if $\Delta$ is a real number with $\Delta \geq d_0$ and $G$ is a graph with maximum degree at most $\Delta$, then $(G,\Hy)$ is fractionally properly conflict-free $(1+\epsilon)\Delta$-colorable for any hypergraph $\Hy$ with $V(\Hy)=V(G)$ and $\rank(\Hy) \leq \Delta$.
\end{theorem}

To prove the theorem, we consider the dual problem of fractional coloring that transforms the problem to finding maximum weighted stable sets with specific properties and is easier to work with.
We state the dual problem and prove the duality in Lemma~\ref{dual}.
The remaining task, Lemma~\ref{weighted_1}, is to construct a desired stable set randomly.  
We first randomly construct an induced subgraph with small maximum degree (and hence with small chromatic number) with a very large fraction of the weight by using concentration inequalities, and then choose a stable set from it that hits a large fraction of the weight.  

\begin{lem} \label{dual}
If $G$ is a graph and $\Hy$ is a hypergraph with $V(\Hy)=V(G)$, then for every positive real number $t$, the following two statements are equivalent.
	\begin{enumerate}
		\item[(1)] $(G,\Hy)$ is fractionally properly conflict-free $t$-colorable. 
		\item[(2)] For any functions $f: V(G) \rightarrow {\mathbb R}_{\geq 0}$ and $g: E(\Hy) \rightarrow {\mathbb R}_{\geq 0}$ with $\sum_{v \in V(G)}f(v)+ \sum_{z \in E(\Hy)} g(z)=1$, there exists a stable set $S$ in $G$ such that $\sum_{v \in S}f(v) + \sum_{z \in E(\Hy), |z \cap S|=1}g(z) \geq \frac{1}{t}$.
	\end{enumerate}
\end{lem}

\begin{lem} 
\label{weighted_1}
For every $\epsilon>0$, there exists an integer $d_0$ such that if $\Delta \geq d_0$, $G$ is a graph with maximum degree at most $\Delta$, $\Hy$ a hypergraph with $V(\Hy)=V(G)$ and $\rank(\Hy) \leq \Delta$, and $f: V(G) \rightarrow {\mathbb R}_{\geq 0}$ and $g: E(\Hy) \rightarrow {\mathbb R}_{\geq 0}$ are functions with $\sum_{v \in V(G)}f(v)+\sum_{z \in E(\Hy)}g(z)=1$, then there exists a stable set $S$ of $G$ such that $\sum_{v \in S}f(v) + \sum_{z \in E(\Hy), |z \cap S|=1}g(z) \geq \frac{(1-\epsilon)^2}{(1+2\epsilon)\Delta}$.
\end{lem}

Before proving Lemmas~\ref{dual} and~\ref{weighted_1}, we prove Theorem \ref{fractional} assuming the lemmas.

\begin{proof}[Proof of Theorem \ref{fractional}]
To show $(G,\Hy)$ is fractionally properly conflict-free
$(1+\epsilon)\Delta$-colorable, by Lemma~\ref{dual}, it suffices, given
functions $f: V(G) \rightarrow {\mathbb R}_{\geq 0}$ and $g: E(\Hy) \rightarrow
{\mathbb R}_{\geq 0}$ with $\sum_{v \in V(G)}f(v)+\sum_{z \in E(\Hy)}g(z)=1$, to
show there is a stable set $S$ of $G$ with $\sum_{v \in S}f(v) + \sum_{z \in E(\Hy), |z \cap S|=1}g(z) \geq \frac{1}{(1+\epsilon)\Delta}$.

By Lemma~\ref{weighted_1}, for every $\epsilon_0>0$, there exists an integer $d_0$ such that if $\Delta\ge d_0$, then there exists a stable set $S$ in $G$ with $\sum_{v \in S}f(v) + \sum_{z \in E(\Hy), |z \cap S|=1}g(z) \geq \frac{(1-\epsilon_0)^2}{(1+2\epsilon_0)\Delta}$.
Since $\lim_{x \to 0}\frac{1+2x}{(1-x)^2} = 1$, there exists $\epsilon_0>0$
such that $\frac{1+2\epsilon_0}{(1-\epsilon_0)^2} \leq 1+\epsilon$.
Let $d_0$ be the constant mentioned in Lemma~\ref{weighted_1} for $\epsilon_0$.
Now we are done by Lemma~\ref{weighted_1}.
\end{proof}

\begin{proof}[Proof of Lemma~\ref{dual}]
We begin by formulating fractional proper conflict-free coloring as a linear
program.
Let \Emph{$A_1$} be a matrix with rows indexed by $V(G)$ and columns indexed by the set of all stable sets in $G$, such that for each $v \in V(G)$ and stable set $S$ in $G$, the entry of $A_1$ in the $v$-th row and $S$-th column equals 1 if $v \in S$ and equals 0 otherwise.  
Let \Emph{$A_2$} be a matrix with rows indexed by $E(\Hy)$, and columns indexed by the set of all stable sets in $G$, in the same order as in $A_1$, and for each $z \in E(\Hy)$ and stable set $S$ in $G$, the entry of $A_2$ in the $z$-th row and $S$-th column equals 1 if $|z \cap S|=1$ and equals 0 otherwise.  
Let \Emph{$A$} be the matrix with $|V(G)|+|E(\Hy)|$ rows such that its first $|V(G)|$ rows form $A_1$ and its last $|E(\Hy)|$ rows form $A_2$.
In the rest of proof, we will frequently denote by 1 a vector all of whose entries are 1.

It is easy to see that $(G,\Hy)$ is fractionally properly conflict-free $t$-colorable if and only if there exists a nonnegative rational vector $x$ with $1^Tx \leq t$ and $Ax \geq 1$, where $x\in[0,1]^{|\mathcal{I}(G)|}$, where $\mathcal{I}(G)$ is the set of
all independent sets of $G$.  
Since $A$ is an integral matrix, the fractional proper conflict-free chromatic number of $(G,\Hy)$ equals $\min_x 1^Tx$ over nonnegative real vectors $x$ with $Ax \geq 1$; moreover, this minimum is attained by a nonnegative rational vector $x$.

Now we prove that (1) implies (2).  
Assume that the fractional proper conflict-free chromatic number of $(G,\Hy)$ is at most $t$.
So there exist positive integers $a,b$ with $a/b \leq t$ and a proper\aside{$t$, $a$, $b$} conflict-free $(a:b)$-coloring $\varphi$ of $(G,\Hy)$.  
Let $f,g$ be functions as given in (2).  
Since $\varphi$ is a proper conflict-free\aside{$f$, $g$, $\vph$} $(a:b)$-coloring of $(G,\Hy)$, by the definition of $(a:b)$-coloring,
	\begin{align*}
		\sum_{i=1}^a\left(\sum_{v \in \varphi^{-1}(i)}f(v) + \sum_{\substack{z \in E(\Hy)\\ |z \cap \varphi^{-1}(i)|=1}}g(z)\right) & \geq b\sum_{v \in V(G)}f(v) + \sum_{i=1}^a\sum_{\substack{z \in E(\Hy)\\ |z \cap \varphi^{-1}(i)|=1}}g(z) \\
		& \geq b\sum_{v \in V(G)}f(v) + b\sum_{z \in E(\Hy)}g(z) = b.
	\end{align*}

By the pigeonhole principle, there exists $i \in [a]$ with $\sum_{v \in \varphi^{-1}(i)}f(v) + \sum_{z \in E(\Hy), |z \cap \varphi^{-1}(i)|=1}g(z) \geq \frac{b}{a} \geq \frac{1}{t}$.  
So (2) holds, since $\varphi^{-1}(i)$ is a stable set.

Now we prove that (2) implies (1).  
Assume that (2) holds.  
Suppose to the contrary that the fractional proper conflict-free chromatic number of $(G,\Hy)$ is greater than \Emph{$t$}.  
By the duality theorem of linear programming, there exist nonnegative functions $f_1: V(G) \rightarrow {\mathbb R}_{\geq 0}$ and $g_1: E(\Hy) \rightarrow {\mathbb R}_{\geq 0}$\aside{$f_1$, $g_1$} such that $\sum_{v \in V(G)}f_1(v)+ \sum_{z \in E(\Hy)}g_1(z)>t$, and for every stable set $S$ in $G$, we have $\sum_{v \in S}f_1(v) + \sum_{z \in E(\Hy), |z \cap S|=1}g_1(z) \leq 1$.  
Let $s := \sum_{v \in V(G)}f_1(v)+\sum_{z \in E(\Hy)}g_1(z)$.\aside{$s$}
Note that $s>t$.  Let $f$ and $g$ be the functions such that
$f:=\frac{1}{s} \cdot f_1$ and $g:=\frac{1}{s} \cdot g_1$.\aside{$f$, $g$}
Hence $\sum_{v \in V(G)}f(v)+\sum_{z \in E(\Hy)}g(z) = \frac{1}{s} \cdot (\sum_{v \in V(G)}f_1(v)+\sum_{z \in E(\Hy)}g_1(z)) = 1$, and for every stable set $S$ in $G$, we have $\sum_{v \in
S}f(v) + \sum_{z \in E(\Hy), |z \cap S|=1}g(z) = \frac{1}{s} \cdot (\sum_{v \in S}f_1(v) + \sum_{z \in E(\Hy), |z \cap S|=1}g_1(z)) \leq \frac{1}{s} < \frac{1}{t}$, contradicting (2).
\end{proof}

To prove Lemma~\ref{weighted_1}, the second main step in our plan, we need the following three lemmas to bound various probabilities.  
The first is the Chernoff Bound, which is well-known (proofs are available in most probability textbooks).  
The second and third are straightforward applications of elementary calculus, so we defer their proofs to the appendix.

\begin{lem}[Chernoff bound] \label{chernoff}
Let $X_1,\ldots,X_n$ be i.i.d.\ random variables such that for every $i \in[n]$, we have $X_i=1$ with probability $p$ and $X_i=0$ with probability $1-p$.
Let $X:=\sum_{i=1}^nX_i$.  For every $\delta$ with $0<\delta<1$, we have $P(|X-{\mathbb E}[X]| \geq \delta {\mathbb E}[X]) \leq 2e^{-\delta^2{\mathbb E}[X]/3}$.
\end{lem}

\begin{lem} \label{x(1-p)x}
Let $p$ be a real number with $0<p<1$.  Let $a,b$ be real numbers.
If $f(x):=x(1-p)^x$ for every real number $x$,
then $f(x) \geq \min\{f(a),f(b)\}$ for every $x$ with $a \leq x \leq b$.
\end{lem}

\begin{lem} \label{limit_x1-logxx}
$\lim_{x \to \infty} x(1- \frac{\log x}{x})^x = 1$.
\end{lem}

Now we prove Lemma~\ref{weighted_1}, completing the proof of Theorem \ref{fractional}.
For convenience, we restate it.

\noindent
\textbf{Lemma~\ref{weighted_1}.}~%
\textit{%
For every $\epsilon>0$, there exists an integer $d_0$ such that if $\Delta \geq d_0$, $G$ is a graph with maximum degree at most $\Delta$, $\Hy$ a hypergraph with $V(\Hy)=V(G)$ and $\rank(\Hy) \leq \Delta$, and $f: V(G) \rightarrow {\mathbb R}_{\geq 0}$ and $g: E(\Hy) \rightarrow {\mathbb R}_{\geq 0}$ are functions with $\sum_{v \in V(G)}f(v)+\sum_{z \in E(\Hy)}g(z)=1$, then there exists a stable set $S$ of $G$ such that $\sum_{v \in S}f(v) + \sum_{z \in E(\Hy), |z \cap S|=1}g(z) \geq
\frac{(1-\epsilon)^2}{(1+2\epsilon)\Delta}$.}

\begin{proof}
Fix $\epsilon>0$.\aside{$\epsilon$}  Note that $\lim_{x \to \infty}x^{-\epsilon^2/3}=0=\lim_{x \to \infty}\frac{\log x}{x}$ and recall (from Lemma \ref{limit_x1-logxx}), that $\lim_{x \to \infty} x(1- \frac{\log x}{x})^x = 1$.  
So there exists an integer \Emph{$d_0$} such that for every real number $d$, if $d \geq d_0$, then $\epsilon\ge \max\{2d^{-\epsilon^2/3}, \frac{\log d}{d}, 1- d(1- \frac{\log d}{d})^d, 1/\log d\}$. 
Let $\Delta$, $G$, $\Hy$, and functions $f,g$\aside{$\Delta$, $G$, $\Hy$, $f$, $g$} be as prescribed in the lemma.  
Let $p := \frac{\log \Delta}{\Delta}$.\aside{$p$}  

Let \Emph{$A$} be the subset of $V(G)$ obtained by, for each vertex $v$ of $G$, independently putting $v$ into $A$ with probability $p$.
For each $v \in V(G)$, let \Emph{$X_v$} be the random variable with $X_v=1$ if $v \in A$, and $X_v=0$ otherwise.
For each $v \in V(G)$, let \Emph{$Y_v$} be the random variable such that $Y_v=1$ if
$X_v=1$ and there are more than $(1+\epsilon)p\Delta$ neighbors $w$ of $v$ with
$X_w=1$; otherwise, $Y_v=0$.
Let $B := \{v \in V(G): X_v=1,Y_v=0\}$.\aside{$B$}
Hence 
$$
{\mathbb E}[\sum_{v \in B}f(v) + \sum_{\substack{z \in E(\Hy)\\ |z \cap B|=1}}g(z)] = 
\sum_{v \in V(G)}f(v) \cdot P(v \in B) + \sum_{z \in E(\Hy)} g(z) \cdot P(|z \cap B|=1).
$$

For every $v \in V(G)$, by the Chernoff bound (Lemma \ref{chernoff}), 

\begin{align*}
P(Y_v=1) &= P(X_v=1\mbox{ and }\sum_{w \in N(v)}X_w>(1+\epsilon)p\Delta) \\
&=  P(X_v=1) \cdot P\left(\sum_{w \in
N(v)}X_w>(1+\epsilon)\frac{\Delta}{\deg(v)} \cdot {\mathbb E}[\sum_{w \in
N(v)}X_w]\right) \\
&\leq  p \cdot 2e^{-\frac{1}{3}((1+\epsilon)\frac{\Delta}{\deg(v)}-1)^2{\mathbb E}[\sum_{w \in N(v)}X_w]} \\
&=  2p \cdot e^{-\frac{1}{3}((1+\epsilon)\frac{\Delta}{\deg(v)}-1)^2 \cdot p \cdot \deg(v)} \\
&=  2p \cdot e^{-\frac{1}{3}((1+\epsilon)\frac{\Delta}{\sqrt{\deg(v)}}-\sqrt{\deg(v)})^2 \cdot p} \\
& \leq 2p \cdot e^{-\frac{1}{3}(\epsilon\frac{\Delta}{\sqrt{\deg(v)}})^2 \cdot p} \\
& \leq 2p \cdot e^{-\frac{1}{3}\epsilon^2\Delta p} = 2pe^{-\frac{1}{3}\epsilon^2
\log\Delta} = 2p \cdot \Delta^{-\epsilon^2/3}. 
\end{align*}

Since $\Delta \geq d_0$, by our choice of $d_0$, we know $2\Delta^{-\epsilon^2/3} \le \epsilon$.
So $P(Y_v=1) \le p\epsilon$.
Hence, for every $v \in V(G)$, we know $P(v \in B) = P(X_v=1,Y_v=0) =
P(X_v=1)-P(X_v=1,Y_v=1) = P(X_v=1)-P(Y_v=1) \geq p-p\epsilon = p \cdot (1-\epsilon)$.

For each $z \in E(\Hy)$, note that $P(|z \cap B|=1) \geq \sum_{u \in z}P(X_u=1, Y_u=0, X_w=0$ for every $w \in z\setminus\{u\})$.
For each vertex $z \in E(\Hy)$ and $u \in z$, by the Chernoff bound, 
\begin{align*}
	& ~~~~P(Y_u=1|X_u=1\mbox{ and }X_w=0 {\rm \ for \ every \ } w \in z-\{u\}) \\
	& = P(\sum_{w \in N(u)\setminus z} X_w > (1+\epsilon)p\Delta) \\
	& = P(\sum_{w \in N(u)\setminus z} X_w > (1+\epsilon)\frac{\Delta}{|N(u) \setminus z|} {\mathbb E}[\sum_{w \in N(u)\setminus z} X_w]) \\
	& \leq 2e^{-\frac{((1+\epsilon)\frac{\Delta}{|N(u)\setminus z|}-1)^2}{3}{\mathbb E}[\sum_{w \in N(u)\setminus z} X_w]} \\
	& = 2e^{-\frac{((1+\epsilon)\frac{\Delta}{|N(u)\setminus z|}-1)^2}{3}p|N(u)\setminus z|} \\
	& = 2e^{-((1+\epsilon)\frac{\Delta}{\sqrt{|N(u)\setminus z|}}-\sqrt{|N(u)\setminus z|})^2 \cdot \frac{p}{3}} \\
	& \leq 2e^{-(\epsilon\frac{\Delta}{\sqrt{|N(u)\setminus z|}})^2 \cdot \frac{p}{3}} \\
	& \leq 2e^{-\epsilon^2 \Delta p/3} = 2e^{-\epsilon^2 \log\Delta /3} = 2\Delta^{-\epsilon^2/3} \leq \epsilon.
\end{align*}
So $P(Y_u=0|X_u=1, X_w=0$ for every $w \in z-\{u\}) \geq 1-\epsilon$. 
For each $z \in E(\Hy)$ and $u \in z$, we have $P(X_u=1, X_w=0$ for every $w \in z-\{u\}) = p(1-p)^{|z|-1} \geq p(1-p)^{|z|}$. 
Hence, for each $z \in E(\Hy)$ and $u \in z$, 
\begin{align*}
	& ~~~~P(X_u=1\mbox{ and }Y_u=0\mbox{ and } X_w=0 {\rm \ for \ all \ } w \in z\setminus\{u\}) \\
	& = P(X_u=1\mbox{ and }X_w=0 {\rm \ for \ all \ } w \in z\setminus\{u\}) \cdot
P(Y_u=0|X_u=1\mbox{ and }X_w=0 {\rm \ for \ all \ } w \in z\setminus\{u\}) \\
	& \geq p(1-p)^{|z|} \cdot (1-\epsilon). 
\end{align*}
All $z \in E(\Hy)$ satisfy $1 \leq |z| \leq \Delta$, so Lemma~\ref{x(1-p)x} gives $|z|(1-p)^{|z|} \geq \min\{1-p,\Delta(1-p)^{\Delta}\}$.  
Since $\Delta \geq d_0$, we get $1-p = 1- \frac{\log\Delta}{\Delta} \geq 1-\epsilon$ and $\Delta(1-p)^{\Delta} = \Delta(1-\frac{\log\Delta}{\Delta})^\Delta \geq 1-\epsilon$. 
Thus, by summing the probabilities of some disjoint events, we have
\begin{align*}
	P(|z \cap B|=1) & \geq \sum_{u \in z}P(X_u=1, Y_u=0, X_w=0 {\rm \
for \ every }\ w \in z\backslash\{u\}) \\
	& \geq |z| \cdot p(1-p)^{|z|} \cdot (1-\epsilon) \\ 
	& = (1-\epsilon)p \cdot |z|(1-p)^{|z|} \\ 
	& \geq (1-\epsilon)p \cdot \min\{1-p, \Delta(1-p)^{\Delta}\} \\ 
	& \geq (1-\epsilon)p \cdot (1-\epsilon) = (1-\epsilon)^2p.
\end{align*}

Hence 
\begin{align*}
	{\mathbb E}[\sum_{v \in B}f(v) + \sum_{\substack{z \in E(\Hy)\\ |z \cap B|=1}}g(z)] & = \sum_{v \in V(G)}f(v) \cdot P(v \in B) + \sum_{z \in E(\Hy)} g(z) \cdot P(|z \cap B|=1) \\
	 & \geq p(1-\epsilon) \sum_{v \in V(G)}f(v) + (1-\epsilon)^2p \sum_{z \in E(\Hy)}g(z) \\
	 & \geq p(1-\epsilon)^2 (\sum_{v \in V(G)}f(v)+ \sum_{z \in E(\Hy)}g(z)) 
	  = p(1-\epsilon)^2.
\end{align*}
So there exists $B^* \subseteq V(G)$ such that $G[B^*]$ has maximum degree at most $(1+\epsilon)p\Delta$ and $$\sum_{v \in B^*}f(v) + \sum_{\substack{z \in E(\Hy)\\ |z \cap B^*|=1}}g(z) \geq p(1-\epsilon)^2.$$

Since $G[B^*]$ has maximum degree at most $(1+\epsilon)p\Delta$, $G[B^*]$ is
properly $(\lfloor (1+\epsilon)p\Delta \rfloor+1)$-colorable.
Hence $B^*$ is a union of disjoint stable sets $S_1,S_2,...,S_{\lfloor (1+\epsilon)p\Delta \rfloor+1}$ in $G$.  
Note that for every $z \in E(\Hy)$, if $|z \cap B^*|=1$, then $|z \cap S_j|=1$ for some $j$.
So $$\sum_{j=1}^{\lfloor (1+\epsilon)p\Delta \rfloor+1}\sum_{\substack{z \in E(\Hy)\\ |z \cap S_j|=1}}g(z) \geq \sum_{\substack{z \in E(\Hy)\\ |z \cap B^*|=1}}g(z).$$
Hence 
\begin{align*}
	\sum_{j=1}^{\lfloor (1+\epsilon)p\Delta \rfloor+1}(\sum_{v \in S_j}f(v) + \sum_{\substack{z \in E(\Hy)\\ |z \cap S_j|=1}}g(z)) & \geq \sum_{v \in B^*}f(v) + \sum_{\substack{z \in E(\Hy)\\ |z \cap B^*|=1}}g(z) \\
	&\geq  p(1-\epsilon)^2.
\end{align*}

Therefore, there exists a subset $S$ of $B^*$ such that $S$ is a stable set in $G$ and 
\begin{align*}
	\sum_{v \in S}f(v) + \sum_{\substack{z \in E(\Hy)\\ |z \cap S|=1}}g(z) \geq & \frac{1}{\lfloor (1+\epsilon)p\Delta \rfloor+1} \cdot p(1-\epsilon)^2 \\
	\geq & \frac{1}{(1+2\epsilon)p\Delta} \cdot p(1-\epsilon)^2 = \frac{(1-\epsilon)^2}{(1+2\epsilon)\Delta}.
\end{align*}
\aftermath
\end{proof}

Finally, we prove Corollary~\ref{fractional_cor}.
We restate it here.

\medskip

\noindent{\bf Corollary~\ref{fractional_cor}.} {\it For every $\epsilon>0$, there exists $d_0$ such that if $\Delta \geq d_0$ and
$G$ is a graph with maximum degree at most $\Delta$, then there exists a proper
$(a:b)$-coloring $\vph$ of $G$ for some positive integers $a$ and $b$ with $a
\leq (1+\epsilon)\Delta b$ such that
\begin{enumerate}
	\item $\vph$ is a fractionally proper conflict-free $(a:b)$-coloring of $G$,
	\item for any set $C$ of colors with size less than $5b/2$, every component of the subgraph of $G$ induced by the vertices that use only colors in $C$ has at most two vertices, and
	\item $|\vph(v) \cap \vph(w)| \leq b/2$ for any distinct vertices $v,w$ of $G$.
\end{enumerate}}

\begin{proof}
Fix a graph $G$.  Let $\Hy$ be the hypergraph with $V(\Hy)=V(G)$ and $E(\Hy)$
consists of all nonempty subsets of $V(G)$ with size at most $\Delta$.
By Theorem \ref{fractional}, there exists a proper $(a:b)$-coloring $\vph$ of
$(G,\Hy)$ for some positive integers $a$ and $b$ with $a \leq (1+\epsilon)\Delta b$.
So for every non-isolated vertex $v$ of $G$, since $|N_G(v)| \leq \Delta(G)
\leq \Delta$, we have $N_G(v) \in E(\Hy)$, so at least $b$ colors appear
exactly once on $N_G(v)$.  This proves Statement 1.

For any two distinct vertices $v,w$ of $G$, since $|\{v,w\}| \leq \Delta$, by the definition of $E(\Hy)$ 
some edge of $\Hy$
is precisely $\{v,w\}$.  Since the coloring of $\Hy$ is conflict-free, by definition at
least $b$ colors appear exactly once on $\{v,w\}$. So $|\vph(v) \setminus
\vph(w)| + |\vph(w) \setminus \vph(v)| \geq b$, and hence $|\vph(v) \cap
\vph(w)| = (|\vph(v)|+|\vph(w)| - (|\vph(v) \setminus \vph(w)| + |\vph(w)
\setminus \vph(v)|))/2 \leq b/2$.  This proves Statement 3.

Now we prove Statement 2.
Suppose to the contrary that there exists a subset $S$ of $V(G)$ with $|S|=3$ such that $\bigcup_{v \in S}\vph(v) = C$ of size less than $5b/2$ and the subgraph of $G$ induced on $S$ is connected.
So some vertex $x$ in $S$ is adjacent to the other two vertices $y$ and $z$ in $S$.
By Statement 3, $|\vph(y) \cup \vph(z)| = |\vph(y)|+|\vph(z)|-|\vph(y) \cap \vph(z)| \geq 3b/2$. 
Since $\vph$ is proper, $\vph(x) \cap (\vph(y) \cup \vph(z)) = \emptyset$, so $|\vph(x) \cup \vph(y) \cup \vph(z)| = |\vph(x)|+|\vph(y) \cup \vph(z)| \geq 5b/2$, a contradiction.  
This proves Statement 2.
\end{proof}

\section*{Acknowledgments}
Thanks to an anonymous referee whose helpful suggestions shortened and simplified Section~\ref{sec2}.
Thanks also to Louis Esperet for helpful comments on an early draft of this paper.

\bibliographystyle{habbrv}
{\footnotesize{\bibliography{conflict-free-proper}}}

\section*{Appendix: 3 Omitted Proofs}
Here we include 3 proofs that we omitted from the body of the text.

\subsection{Proof of Lemma~\ref{lem:clm1_simple}}

We prove Lemma~\ref{lem:clm1_simple} by considering the range of possible values
of $d$.

\begin{lem}
\label{clm0-lem}
If $3\le {d}\le 9$ and $\beta \geq 0.6R$ and $R \geq 750$, then $$\sum_{i=1}^{\left\lfloor {d}/2\right\rfloor}E_i(d)\beta^{i-d+1}\le R^{-1/2}.$$
\end{lem}

\begin{proof}
Since $d \geq 3$ and $\beta \geq 0.6R$, $\sum_{i=1}^{\left\lfloor {d}/2\right\rfloor}E_i(d)\beta^{i-d+1}\le \sum_{i=1}^{\left\lfloor {d}/2\right\rfloor}E_i(d)(0.6R)^{i-d+1}$.
Note that $E_1(d)=1$, $E_i(2i) = \frac{(2i)!}{i!2^i}$ (by the second bound in Lemma \ref{two_basic_bound}), and $E_i(j) \leq E_i(k)$ for every $j \leq k$.
When $d=3$, $\sum_{i=1}^{\left\lfloor {d}/2\right\rfloor}E_i(d)(0.6R)^{i-d+1} = (0.6R)^{-1} \leq R^{-1/2}$.
When $d =4$, 
	\begin{align*}
		\sum_{i=1}^{\left\lfloor {d}/2\right\rfloor}E_i(d)(0.6R)^{i-d+1} = (0.6R)^{-2} + E_2(4)(0.6R)^{-1} \leq (1+\frac{4!}{2!2^2})(0.6R)^{-1}=4(0.6R)^{-1} \leq R^{-1/2}.
	\end{align*}
	
Since $E_i(d)$ is the number of partitions of the set $[d]$ into $i$ parts with extra properties, $E_i(d) \leq i^d$.
When $5 \leq d \leq 6$, 
\begin{align*}
\sum_{i=1}^{\left\lfloor {d}/2\right\rfloor}E_i(d)(0.6R)^{i-d+1} &\leq (0.6R)^{2-d} + E_2(d)(0.6R)^{3-d} + (d-5)E_3(6)(0.6R)^{4-d} \\
& \leq (0.6R)^{2-d} + 2^d(0.6R)^{3-d} + (d-5)\frac{6!}{3!2^3}(0.6R)^{4-d} \\ &\leq (1+2^6)(0.6R)^{-2} + 15(d-5)(0.6R)^{4-d}.
\end{align*}
So if $d=5$, then $\sum_{i=1}^{\left\lfloor {d}/2\right\rfloor}E_i(d)(0.6R)^{i-d+1} \leq 181/R^2 \leq R^{-1/2}$, since $R \geq 750$; if $d=6$, then $\sum_{i=1}^{\left\lfloor {d}/2\right\rfloor}E_i(d)(0.6R)^{i-d+1} \leq 181R^{-2} + 15(0.6R)^{-2} \leq R^{-1/2}$.

	When $7 \leq d \leq 8$, 
\begin{align*}
\sum_{i=1}^{\left\lfloor {d}/2\right\rfloor}E_i(d)(0.6R)^{i-d+1} & \leq (0.6R)^{2-d} + E_2(d)(0.6R)^{3-d} + E_3(d)(0.6R)^{4-d} + E_4(d)(0.6R)^{5-d} \\ 
& \leq (0.6R)^{-5} + 2^8(0.6R)^{-4} + 3^8(0.6R)^{-3} + \frac{8!}{4!2^4}(0.6R)^{-2} \\
& \leq (1+256+6561)(0.6R)^{-3} + 105(0.6R)^{-2} \\
& \leq 31565R^{-3} + 300R^{-2} = (31565R^{-2.5}+300R^{-1.5})R^{-1/2} \\
& \leq (31565 \cdot 750^{-2.5}+300 \cdot 750^{-1.5})R^{-1/2} \leq R^{-1/2}.
\end{align*}
When $d=9$, 
\begin{align*}
\sum_{i=1}^{\left\lfloor {d}/2\right\rfloor}E_i(d)(0.6R)^{i-d+1} & = (0.6R)^{-7} + E_2(9)(0.6R)^{-6} + E_3(9)(0.6R)^{-5} + E_4(9)(0.6R)^{-4} \\
& \leq 4 \cdot 4^9 (0.6R)^{-4} \leq 4^{10} \cdot (5/3)^4 \cdot 750^{-3.5} \cdot R^{-1/2} \leq R^{-1/2}.
\end{align*}
\aftermath
\end{proof}

\begin{lem}
\label{clm1a-lem}
If $10\le {d}\le \beta^{2/3}$ and $\beta \geq 0.6R \geq 14$, then 
	$$\sum_{i=1}^{\left\lfloor {d}/2\right\rfloor}E_i(d)\beta^{i-d+1}\le R^{-1/2}.$$
\end{lem}

\begin{proof}
Since $10 \leq d \leq \beta^{2/3}$, we have $d/\beta \leq \beta^{-1/3} \leq 1/2$, so
$$\sum_{i=1}^{\left\lfloor {d}/2\right\rfloor}E_i(d)\beta^{i-d+1} \leq \beta
\frac{({d}/\beta)^{\left\lceil {d}/2\right\rceil}}{1-\frac{d}{\beta}} \leq
2\beta (\beta^{-\frac{1}{3}})^{\frac{d}{2}} = 2\beta^{1-\frac{d}{6}} \leq
2\beta^{-2/3} \le \left(\frac{\beta}{0.6}\right)^{-1/2}\le R^{-1/2},$$
where the first inequality follows from Lemma~\ref{S_2_simple_bound_proof}, and
the final two inequalities holds because $\beta \geq 14$ and because $\beta\ge 0.6R$.
\end{proof}

\begin{lem}
\label{clm1b-lem}
If $\beta^{2/3}\le {d}\le \beta^{19/20}$ and $\beta \geq \max\{0.6R,600\}$ and $R \geq 750$, then 
	$$\sum_{i=1}^{\left\lfloor {d}/2\right\rfloor}E_i(d)\beta^{i-d+1}\le R^{-1/2}.$$
\end{lem}

\begin{proof}
Since $\beta \geq 600$, we have $d/\beta \leq \beta^{-1/20} \leq 600^{-1/20}
\leq 3/4$, so by Lemma~\ref{S_2_simple_bound_proof},
	\begin{align*}
		\sum_{i=1}^{\left\lfloor {d}/2\right\rfloor}E_i(d)\beta^{i-d+1} & \le \beta \frac{({d}/\beta)^{\left\lceil {d}/2\right\rceil}}{1-\frac{d}{\beta}}\le 4\beta (\beta^{-\frac{1}{20}})^{{d}/2} = 4\beta^{1-\frac{d}{40}} \leq 4 \cdot (0.6R)^{1-\frac{d}{40}}. 
	\end{align*}
	Since $\beta \geq 600$, we have $d \geq \beta^{2/3} \geq 71$, so $4 \cdot (0.6R)^{1-d/40} \leq 4 \cdot (0.6R)^{-31/40} \leq 6R^{-31/40} \leq 6 \cdot 750^{-11/40} \cdot R^{-1/2}\le R^{-1/2}$, where the penultimate inequality follows from $R \geq 750$.
\end{proof}

\begin{cor}
Lemma~\ref{lem:clm1_simple} is true. 
\end{cor}

\begin{proof}
This follows immediately from Lemmas~\ref{clm0-lem}, \ref{clm1a-lem}, and~\ref{clm1b-lem}.
\end{proof}

\subsection{Proofs of Lemmas~\ref{x(1-p)x} and~\ref{limit_x1-logxx}}

\noindent \textbf{\Cref{x(1-p)x}.}
Let $p$ be a real number with $0<p<1$.  Let $a,b$ be real numbers.
If $f(x):=x(1-p)^x$ for every real number $x$,
then $f(x) \geq \min\{f(a),f(b)\}$ for every $x$ with $a \leq x \leq b$.

\begin{proof}
Let $t := -1/\log(1-p)$.
Since $\frac{{\rm d}f(x)}{{\rm d}x} = (1-p)^x+x(1-p)^x\log(1-p) =
(1-p)^x(1+x\log(1-p))$, $f(x)$ is increasing when $x \leq t$, and $f(x)$ is
decreasing when $x \geq t$.  Hence $f(x) \geq \min\{f(a),f(b)\}$ for every $x$
with $a \leq x \leq b$.
\end{proof}

\noindent \textbf{\Cref{limit_x1-logxx}.}
$\lim_{x \to \infty} x(1- \frac{\log x}{x})^x = 1$.

\begin{proof}
This is a straightforward application of L'H\^{o}spital's rule. 
We have
\begin{align*}
&~~~~\lim_{x\to\infty}\log\left(x\left(1-\frac{\log x}x\right)^x\right)= \lim_{x\to\infty}\left(\log(x)+x\log\left(1-\frac{\log x}x\right)\right) \\
&= \lim_{x\to\infty}\left(\frac{\log(x)}x+\log\left(1-\frac{\log x}x\right)\right)/x^{-1}\\
&\stackrel{\text{H}}{=}
 \lim_{x\to\infty}\left(\frac1{x^2} -\frac{\log x}{x^2} +\left(1-\frac{\log(x)}x\right)^{-1}\left(\frac{-1+\log x}{x^2}\right)\right)/(-x^{-2})\\
&=\lim_{x\to\infty}\left(-1+\log x-\frac{x(\log x-1)}{x-\log x}\right)=\lim_{x\to\infty}\frac{\log x-\log^2x}{x-\log x} =0.
\end{align*}
So 
\begin{align*}
&~~~~\lim_{x\to\infty}x\left(1-\frac{\log x}x\right)^x 
= \lim_{x\to\infty}\exp\left(\log\left(x\left(1-\frac{\log x}x\right)^x\right)\right)=1 
\end{align*}
\aftermath
\end{proof}

\end{document}